\documentclass[11pt]{amsart}

\usepackage{amscd}
\usepackage{amsmath, amssymb, comment}
\usepackage{amsfonts}
\usepackage{enumerate}
\usepackage{color}
\usepackage{url}
\newcommand{\de}{\partial}

\newcommand{\dbar}{\overline{\partial}}

\newcommand{\ddbar}{\sqrt{-1} \partial \overline{\partial}}

\newcommand{\ov}[1]{\overline{#1}}

\newcommand{\ti}[1]{\widetilde{#1}}
\newcommand{\vp}{\varphi}

\newcommand{\ve}{\varepsilon}
\newcommand{\e}{\varepsilon}

\newcommand{\p}{\partial}
\newcommand{\pbar}{\overline{\partial}}
\renewcommand{\e}{\varepsilon}

\renewcommand{\leq}{\leqslant}
\renewcommand{\geq}{\geqslant}

\newcommand{\be}{\begin{equation}}
\newcommand{\ee}{\end{equation}}

\newcommand{\PSH}{\mathrm{PSH}}
\newcommand{\mSH}{\mathrm{mSH}}

\newcommand{\R}{\mathbb{R}}
\newcommand{\C}{\mathbb{C}}
\newcommand{\MA}{\mathrm{MA}}

\begin{document}
\newcounter{remark}
\newcounter{theor}
\setcounter{theor}{1}
\newtheorem{claim}{Claim}
\newtheorem{theorem}{Theorem}[section]
\newtheorem{lemma}[theorem]{Lemma}
\newtheorem{corollary}[theorem]{Corollary}
\newtheorem{proposition}[theorem]{Proposition}
\newtheorem{question}{question}[section]
\newtheorem{definition}[theorem]{Definition}
\newtheorem{remark}[theorem]{Remark}

\numberwithin{equation}{section}

\title[Degenerate elliptic equations]{Fully non-linear degenerate elliptic equations in complex geometry}

\begin{abstract}
We derive an {\em a priori} real Hessian estimate for solutions of a large family of geometric fully non-linear elliptic equations on compact Hermitian manifolds, which is independent of a lower bound for the right-hand side function. This improves on the estimates of Sz\'ekelyhidi \cite{Szekelyhidi18} and additionally applies to elliptic equations with a degenerate right-hand side. As an application, we establish the optimal $C^{1,1}$ regularity of envelopes of $(\theta, m)$-subharmonic functions on compact Hermitian manifolds.
\end{abstract}

\author[J. Chu]{Jianchun Chu}
\address{Department of Mathematics, Northwestern University, 2033 Sheridan Road, Evanston, IL 60208}
\email{jianchun@math.northwestern.edu}

\author[N. McCleerey]{Nicholas McCleerey}
\address{Department of Mathematics, University of Michigan, Ann Arbor, 530 Church St, Ann Arbor, MI 48109}
\email{njmc@umich.edu}


\maketitle

\section{Introduction}

Suppose that $(X^n,\omega)$ is a compact Hermitian manifold of complex dimension $n$ without boundary. Let $g$ denote the Riemannian metric corresponding to $\omega$. We will additionally fix a real $(1,1)$-form $\chi_0$ on $X$. For any $u\in C^2(X)$, we will look at the form:
\[
\chi := \chi_0 + i\p\pbar u,
\]
and its associated Hermitian endomorphism $A$ on $T^{1,0}X$:
\[
A^p_q := g^{p\ov{j}}\chi_{q\ov{j}}.
\]

Following the set-up of Sz\'ekelyhidi \cite{Szekelyhidi18}, we will be interested in solving elliptic equations of the form:
\begin{equation}\label{main eqn}
F(A) := f(\lambda_1, \ldots, \lambda_n) = h,
\end{equation}
where here $h\in C^{\infty}(X)$ is fixed and $f: \Gamma\rightarrow \R$ is a concave symmetric function of the eigenvalues of $A$ (which we denote by $\lambda_1, \ldots, \lambda_n$), defined on the convex, open, symmetric cone $\Gamma\subsetneq\R^n$. We refer the reader to \cite{Szekelyhidi18} or the beginning of Section 2 for a complete description of the necessary assumptions we will need on $f$ and $\Gamma$.

Equations of the form \eqref{main eqn} were first studied by Caffarelli, Nirenberg, and Spruck, who solved the Dirichlet problem for domains in $\R^n$, in their pioneering paper \cite{CNS85}. Their work was subsequently generalized to compact Riemannian manifolds by Guan \cite{Guan14}, who also introduced the idea of using the existence of certain subsolutions to derive {\em a priori} $C^2$ estimates for \eqref{main eqn}. The complex case was recently studied by Sz\'ekelyhidi \cite{Szekelyhidi18}, using what he calls $\mathcal{C}$-subsolutions, which generalize the subsolutions of \cite{Guan14}. In particular, he shows in \cite{Szekelyhidi18} that the existence of a $\mathcal{C}$-subsolution implies {\em a priori} $C^\infty$-estimates for solutions to \eqref{main eqn}.

The importance of \eqref{main eqn} comes from the fact that it is general enough to simultaneously cover many natural geometric PDEs. The most well-known example is the complex Monge-Amp\`ere equation, famously solved by Yau \cite{Yau78} when $(M,\omega)$ is K\"ahler in his resolution of the Calabi conjecture \cite{Calabi57}. For general Hermitian $(M,\omega)$, the complex Monge-Amp\`{e}re equation was solved by Tosatti-Weinkove \cite{TW10b}, building on several earlier works (see for instance Cherrier \cite{Cherrier87}, Hanani \cite{Hanani96}, Guan-Li \cite{GL10}, Tosatti-Weinkove \cite{TW10a}, Zhang-Zhang \cite{ZZ11}, etc.).

Another well-known example is the complex Hessian equation. When $(M,\omega)$ is K\"{a}hler, Dinew-Ko\l odziej \cite{DK17} proved a Liouville type theorem for $(\omega,m)$-subharmonic functions in $\mathbb{C}^n$, which, when combined with the estimate of Hou-Ma-Wu \cite{HMW10}, solved the complex Hessian equation. When $(M,\omega)$ is Hermitian, the complex Hessian equation was solved by Sz\'ekelyhidi \cite{Szekelyhidi18} and Zhang \cite{Zhang17}, independently.

A third example is the complex Hessian quotient equation (these are equations of the form $f = (\sigma_m/\sigma_\ell)^{\frac{1}{m-\ell}}$, for $\ell < m$ and with $\sigma_m$ being the $m^\text{th}$-symmetric function), which was solved by Sz\'ekelyhidi \cite{Szekelyhidi18} when the right-hand side is constant (assuming the existence of a $\mathcal{C}$-subsolution -- see below). Previous special cases of this equation had been investigated by Song-Weinkove \cite{SW08}, in connection with stable points of the $J$-flow of Donaldson \cite{Do1} and Chen \cite{Ch2}, and also by Fang-Lai-Ma \cite{FLM11}. When the right-hand side is not constant, analogous results were obtained by Sun \cite{Sun13,Sun17a,Sun17b} (see also Li \cite{Li14}, Guan-Sun \cite{GS15}).

The last example of a geometric equation that falls under the purview of \eqref{main eqn} is the Monge-Amp\`ere equation for $(n-1)$-plurisubharmonic functions  (see \cite[Section 7]{Szekelyhidi18}), introduced by Fu-Wang-Wu \cite{FWW10} as a generalization of the Monge-Amp\`ere equation. This equation can be written as:
\begin{equation}\label{n-1 MA}
\left(\omega_{h}+\frac{1}{n-1}((\Delta_{\omega} u)\omega-\ddbar u)\right)^n = e^{h}\omega^{n}
\end{equation}
where $\omega_{h}$ and $\omega$ are Hermitian metrics. When $\omega$ is a K\"ahler metric with non-negative orthogonal bisectional curvature, Fu-Wang-Wu \cite{FWW15} solved \eqref{n-1 MA}. When $\omega$ is a general K\"ahler metric, \eqref{n-1 MA} was solved by Tosatti-Weinkove in \cite{TW17}. They later relaxed this to only requiring that $\omega$ be a Hermitian metric in \cite{TW19}.

\bigskip

Equation \eqref{main eqn} is in general fully non-linear, and as such, strict ellipticity is not guaranteed without some assumption on the right-hand-side function $h$ -- the necessary condition obtained in \cite{Szekelyhidi18} is to require that
\[
\sup_{\p\Gamma} f < h < \sup_\Gamma f,
\]
where here we write:
\[
\sup_{\de\Gamma}f := \sup_{\lambda'\in\de\Gamma}\limsup_{\substack{\lambda\rightarrow\lambda'\\\lambda\in\Gamma}}f(\lambda).
\]
However, a number of geometric applications require that one consider the ``degenerate" elliptic case when $\sup_{\p\Gamma}f = \inf_X h$ -- the most well-known such situation occurs when one studies the space of K\"ahler metrics on a compact K\"ahler manifold $X$. In \cite{Mabuchi87}, Mabuchi introduced a Riemannian metric on this indefinite dimensional space.  Semmes \cite{Semmes92} and Donaldson \cite{Donaldson99} independently showed that the geodesic equation in \cite{Mabuchi87} can be rewritten as a Dirichlet problem for the homogeneous Monge-Amp\`ere equation on $X$ cross an annuls $A\subset \C$. Such geodesics have played a large role in many recent developments in K\"ahler geometry, and so understanding their regularity was an important question, resolved by Chu-Tosatti-Weinkove \cite{CTW17}, who showed that these geodesics are $C^{1,1}$ regular. Their result builds on an extensive body of previous work, including that of Chen \cite{Chen00}, B\l ocki \cite{Blocki12}, and Berman \cite{Berman17}, and is known to be optimal thanks to examples of Lempert-Vivas \cite{LV13}, Darvas-Lempert \cite{DL12}, and Darvas \cite{Da14}. Later, Chu-McCleerey \cite{CM19} generalized this to $C^{1,1}$-regularity of geodesics between K\"ahler metrics on singular K\"ahler varieties (on the smooth locus of the variety).

Another important application requiring degenerate right-hand sides is the study of the envelope:
\[
P(b) := \sup \{ v\in \PSH(X,\omega)\ |\ v\leq b\}
\]
for a smooth obstacle function $b\in C^\infty(X)$ (here $\PSH(X,\omega)$ is the set of all $\omega$-plurisubharmonic functions). It was shown by Tosatti \cite{Tosatti18} and Chu-Zhou \cite{CZ19} that $P(b) \in C^{1,1}(X)$. To see the connection, note that envelope $P(b)$ satisfies a free-boundary problem involving the contact set $K := \{P(b) = b\}$:
\begin{equation}\label{app}
F_{\MA}(\omega _{P(b)}) = \chi_K \left(\frac{\omega_b^n}{\omega^n}\right)^{1/n},
\end{equation}
where here we write $\omega_b := \omega + i\p\pbar b$ and $F_{\MA}(A) := \left(\lambda_1\cdot\ldots\cdot\lambda_n\right)^{1/n}$ for the Monge-Amp\`ere operator. Berman-Demailly \cite[Corollary 2.5]{BD} show that Equation \eqref{app} follows if $P(b)$ has bounded Laplacian (the proof is in fact already basically contained in \cite{Berman09}), which was later shown to be true by Berman \cite{Berman19} when $b\in C^\infty(X)$. Berman's result was later strengthened by Darvas-Rubenstein \cite{DR16} to requiring that $b$ have only bounded Laplacian, and this was subsequentially used by Di Nezza-Trapani \cite{DNT19} to show \eqref{app} for more singular envelopes as well.

\bigskip

To deal with degenerate $h$, the most common method is to approximate $h$ by a family of non-degenerate $h_i$. One then shows that solutions to $F(A) = h_i$ satisfy {\em a priori} estimates independent of $\inf_{X}h$, so they pass to the limiting (degenerate) solution. In this paper, we will derive such an estimate for the real Hessian of $u$ (which we mean to be the Hessian of $u$ with respect to the Levi-Civita connection induced by $g$), which allows us to apply this technique to the degenerate case of Equation \eqref{main eqn}.

As in \cite{Szekelyhidi18}, our estimates will depend on the existence of a $\mathcal{C}$-subsolution to \eqref{main eqn} -- in Section 2, we show that the alternate definition given on \cite[p.345]{Szekelyhidi18} extends easily to the degenerate setting, so the notion of $\mathcal{C}$-subsolution is still well-defined. Moreover, it will be clear from the definition that if $\underline{u}$ is a $\mathcal{C}$-subsolution to \eqref{main eqn}, then there exists a $\sigma_0 > 0$, depending only $(X,\omega)$, $\Gamma$, $f$, $\chi_{0}$, $\underline{u}$ and $h$, such that $\underline{u}$ is a $\mathcal{C}$-subsolution to the non-degenerate equation:
\[
F(A) = h + 2\sigma_0.
\]
In Section 3, we show that $\sigma_0$ can be used to make the estimates independent of $\inf_X h$. Moreover (see Remark \ref{sigma eq 1}), it is easy to see that $\sigma_0$ can always be chosen to be 1 for many geometric PDEs (including the complex Monge-Amp\` ere, the complex Hessian equations and the Monge-Amp\`ere equation for $(n-1)$-plurisubharmonic functions), so that the estimates for these equations also do not depend on $\sigma_0$.

With this background, we can now state our main result:

\begin{theorem}\label{main estimate}
Let $u$ be a smooth solution of \eqref{main eqn} with $\sup_Xu=-1$. Suppose that $\underline{u}$ is a $\mathcal{C}$-subsolution of \eqref{main eqn} and  $\sup_{\de\Gamma}f<h<\sup_{\Gamma}f$, then there exists a constant $C$ depending only on $(X,\omega)$, $\chi_0$, $\underline{u}$, $\sigma_0$, $\sup_{X}h$, $\sup_{X}|\de h|_g$, and a lower bound of $\nabla^2 h$ such that
\[
\sup_{X}|u|+\sup_{X}|\de u|_g+\sup_{X}|\nabla^2 u|_g \leq C,
\]
where $\nabla$ is the Levi-Civita connection of $g$.
\end{theorem}

As mentioned, our main contribution is the real Hessian estimate. If $\de X\neq\emptyset$, our argument gives an interior bound on the real Hessian (see Remark \ref{boundary}). As mentioned previously, there exist examples of smooth boundary data for the Dirichlet problem for the homogenous complex Monge-Amp\`ere equation such that no $C^2$ solution exists. This implies that one should not expect a degenerate solution to \eqref{main eqn} to be better than $C^{1,1}$ if $h$ is degenerate, meaning that our estimate should be optimal (although, to the best of the authors' knowledge, a concrete example on a closed manifold appears to be unknown -- c.f. \cite{Plis05, DPZ19} for related examples).

\bigskip

To prove Theorem \ref{main estimate}, our approach builds off previous such estimates for the Monge-Amp\`ere equation -- in particular, we proceed by trying to control the largest eigenvalue of (a perturbed version of) $\nabla^2 u$ using the maximum principle (cf. \cite{CTW19}). The main task is to control a third order term of the form:
\be\label{bad terms}
\sum_{i}\frac{F^{i\ov{i}}|u_{V_{1}V_{1}i}|^{2}}{\lambda_{1}^{2}},
\ee
where $V_{1}$ is the unit eigenvector corresponding to $\lambda_{1}$ and $F^{i\ov{i}}$ is the $i\ov{i}$-derivative of $F$ ($F(A)=\log(\lambda_{1}\cdot\ldots\cdot\lambda_{n})$). Previous arguments for the Monge-Amp\`ere equation (\cite{CTW19,Chu18}) heavily utilize special properties of the corresponding $F$. More precisely, for the Monge-Amp\` ere equation, both $F^{i\ov{i}}$ and $F^{i\ov{j},k\ov{l}}$ have simple expressions in terms of the solution metric -- in the general context of Equation \eqref{main eqn}, this is no longer true however, so we need to apply new techniques to control \eqref{bad terms}.

Our solution starts with the term:
\begin{equation}\label{good terms}
\sum_{i\neq q}\frac{(F^{q\ov{q}}-F^{i\ov{i}})|\chi_{i\ov{q}V_{1}}|^{2}}{\lambda_{1}(\chi_{i\ov{i}}-\chi_{q\ov{q}})},
\end{equation}
which is non-negative thanks to the concavity of $F$. In order to effectively use this term, our first contribution is to finesse a technique of Hou-Ma-Wu \cite{HMW10} (which is in turn based on ideas developed for the real Hessian equation by Chou-Wang \cite{CW01}), to control the individual summands in \eqref{bad terms} based on the relative sizes of the corresponding $F^{i\ov{i}}$. This extracts more non-negative terms than previous versions of this technique (cf. Hou-Ma-Wu \cite{HMW10}, Sz\'ekelyhidi \cite{Szekelyhidi18}, Chu-Tosatti-Weinkove \cite{CTW19}, etc.), allowing us to control most of the terms in \eqref{bad terms}.

Even then, \eqref{good terms} is not sufficient to control all of \eqref{bad terms}, as it is possible that $F^{q\ov{q}}-F^{i\ov{i}}$ might be quite small. In this case however, we show that the corresponding non-positive terms are also quite small, picking up an extra factor of $\lambda_{1}^{-2}$ compared to the other summands in \eqref{bad terms}. Thus, we show that these terms can be controlled by the addition of an extra quantity to the maximum principle depending on $\nabla^{2}u$, similar to \cite{Chu18}. As a remark, we point out that the extra term in \cite{Chu18} depends on $\de\dbar u$, and that the argument there also depends heavily on the structure of $F_\mathrm{MA}$. In contrast, our new extra term depends on $\nabla^{2}u$, making the argument significantly more delicate.

\bigskip

A direct consequence of Theorem \ref{main estimate} is the $C^{1,1}$-estimate and the existence of $C^{1,1}$ solutions in the degenerate case, provided there exists an approximating family of solutions to the non-degenerate equations:

\begin{theorem}\label{degenerate estimate}
Suppose that $h\in C^2(X)$ satisfies $\sup_{\p\Gamma} f \leq h < \sup_\Gamma f$, and that $h_\e\in C^2(X)$ satisfy $\sup_{\p\Gamma} f < h_{\e} < \sup_\Gamma f$ and $h_\e \rightarrow h$ in $C^2$.  Let $\underline{u}$ be a $\mathcal{C}$-subsolution to the degenerate equation
\[
F(A(u)) = h(x),
\]
and let $u_\e$ be smooth solutions to the non-degenerate equations:
\[
F(A(u_\e)) = h_\e(x)
\]
for all $0 < \e$ sufficiently small. Then there exists a constant $C$, independent of $\e$ and depending only on $(X,\omega)$, $\chi_0$, $\underline{u}$, $\sigma_0$, $\sup_{X}h$, $\sup_{X}|\de h|_g$, and a lower bound of $\nabla^2 h$ such that
\[
\sup_{X}|u_\e|+\sup_{X}|\de u_\e|_g+\sup_{X}|\nabla^2 u_\e|_g \leq C,
\]
where $\nabla$ is the Levi-Civita connection of $g$.

In particular, up to extracting a subsequence, the $u_\e$ converge to a $u\in C^{1,1}(X)$ solving the degenerate version of equation \eqref{main eqn}:
\[
f^*(\lambda(u)) = h(x),
\]
where here $f^*$ is the upper semi-continuous extension of $f$ to $\ov{\Gamma}$.
\end{theorem}

As discussed in \cite{Szekelyhidi18}, the {\em a priori} estimates in Theorem \ref{main estimate} cannot be used to show the existence of solutions to \eqref{main eqn} (e.g. using the continuity method) without further assumptions to guarantee the existence of appropriate $\mathcal{C}$-subsolutions along the continuity path. As such, it is necessary to assume the existence of the solutions $u_\e$ in Theorem \ref{degenerate estimate}. It is quite non-trivial to come up with natural conditions guaranteeing the existence of $\mathcal{C}$-subsolutions -- see, for instance, the well-known conjectures of Lejmi-Sz\'ekelyhidi \cite{LS15} and Sz\'ekelyhidi \cite{Szekelyhidi18} for the complex Hessian quotient equation. For the degenerate complex Hessian equation however, the zero function is automatically a $\mathcal{C}$-subsolution, which allows us to solve degenerate equations without further assumptions:

\begin{corollary}\label{hessian}
Let $(X,\omega)$ be a compact Hermitian manifold and $1 \leq m \leq n$. Let $\theta$ be a (strictly) $m$-positive form on $X$ and write $\mSH(M,\theta)$ for the set of all $(\theta,m)$-subharmonic functions on $X$ with respect to $\omega$ (see Section 5 for definitions). Then for any non-negative function $h$ on $X$ such that $\int_{X}h\omega^n>0$ and $h^{\frac{1}{m}}\in C^{2}(X)$, there exists a pair $(u,c)\in C^{1,1}(X)\times\mathbb{R}_{+}$ such that
\[
\begin{cases}
(\theta+\ddbar u)^{m}\wedge \omega^{n-m} = ch\omega^{n}, \\
u\in\mSH(X,\theta),  \quad \sup_{X}u = -1.
\end{cases}
\]
\end{corollary}

\bigskip

Lastly, in Section 6, we give a geometric application of our result, by showing how to adapt the estimates in Theorem \ref{main estimate} to get regularity of envelopes of $m$-subharmonic functions, generalizing the above results for the Monge-Amp\`ere equation.

\begin{theorem}\label{envelope}
Let $(X,\omega)$ be a compact Hermitian manifold and $\theta$ a (strictly) $m$-positive form. If $b\in C^{1,1}(X)$, we define the envelope:
\[
P_{m,\theta}(b) := \sup\{ v\in\mSH(X, \theta)~|~v\leq b \}.
\]
Then $P_{m,\theta}(b)\in C^{1,1}(X)$. In particular, $P_{m,\theta}(b)$ solves:
\[
(\theta + i\p\pbar P_{m,\theta}(b))^m\wedge\omega^{n-m} = \chi_{K} \theta_b^m\wedge\omega^{n-m},
\]
where $\theta_{b}=\theta+i\p\pbar b$ and $K = \{ P_{m,\theta}(b) = b\}$ is the contact set.
\end{theorem}

We now outline the contents of the rest of the paper. Section 2 recalls some necessary background material, largely along the lines of \cite{Szekelyhidi18}. In section 3, we show that the {\em a priori} estimates of \cite{Szekelyhidi18} depend only on the background data in Theorem \ref{main estimate}. We prove our main result, Theorem \ref{main estimate}, in Section 4.  In Section 5, we prove Theorem \ref{degenerate estimate}, and show how to get Corollary \ref{hessian}. Finally, in Section 6, we apply our estimates to $P_m(b)$ and prove Theorem \ref{envelope}.

{\bf Acknowledgements:} The authors were partially supported by NSF RTG grant DMS-1502632. We would also like to thank Valentino Tosatti for helpful discussions.

\section{Background and Notation}

Recall that $(X^n,\omega)$ is a compact Hermitian manifold, and we write $g$ for the Riemannian metric corresponding to $\omega$. We will also let $\nabla$ denote the Levi-Civita connection induced by $g$.

We begin by recalling all the necessary assumptions required on Equation \eqref{main eqn}. First, we shall suppose that $\Gamma\subsetneq\mathbb{R}^n$ is an open, symmetric, convex cone with vertex at the origin. Additionally, we will require that $\Gamma$ contain the positive orthant:
\[
\begin{split}
&\Gamma_n \subseteq \Gamma, \\
\Gamma_n = \{(v_1,\ldots,v_n)&\in\mathbb{R}^n~|~v_i>0 \ \text{for each $i$}\}.
\end{split}
\]

Now, we will assume that $f: \Gamma \rightarrow \R$ is a smooth function on $\Gamma$ such that
\begin{enumerate}[(i)]
\item $f$ is concave,
\item $f_i=\frac{\de f}{\de\lambda_i}>0$ for each $i$.
\item For any $\sigma < \sup_\Gamma f$ and $\lambda\in\Gamma$, we have $\lim_{t\rightarrow\infty} f(t\lambda) > \sigma$.
\end{enumerate}

Let $\lambda(u)\in\R^n$ be the function taking $u$ to the un-ordered $n$-tuple of eigenvalues of $A$, the Hermitian endomorphism on $T^{1,0}X$ defined in the introduction. Since both $\Gamma$ and $f$ are symmetric, we have that $f(\lambda(u))$ is well-defined and smooth on $X$. We are then interested in the equation:
\be\label{new equation}
\begin{cases}
&f(\lambda(u)) = h(x)\in C^\infty(X)\\
&\sup_{\p\Gamma} f < h < \sup_\Gamma f.
\end{cases}
\ee
We say \eqref{new equation} is a {\bf non-degenerate} equation (we shall also refer to its solutions as being non-degenerate). We will also be interested in the {\bf degenerate} version of equation \eqref{new equation}, which is when $\sup_{\p\Gamma} f = \inf_X h$; more precisely, when we say that \eqref{new equation} is degenerate, we actually refer to the equation:
\be\label{degen equation}
\begin{cases}
&f^*(\lambda(u)) = h(x)\in C^\infty(X)\\
&\sup_{\p\Gamma} f \leq h < \sup_\Gamma f,
\end{cases}
\ee
where $f^*$ is the upper semi-continuous extension of $f$ to $\ov{\Gamma}$. We will also abuse terminology slightly further, and say that solutions to \eqref{degen equation} are {\bf degenerate solutions} to equation \eqref{new equation}. Note that a degenerate solution $u$ will in general only have $\lambda(u)\in\ov{\Gamma}$.

\bigskip

For any $\sigma\in(\sup_{\de\Gamma}f,\sup_{\Gamma}f)$, we write
\[
\Gamma^\sigma = \{ \lambda\in\Gamma ~|~ f(\lambda) > \sigma\}.
\]
It is easy to check that $\Gamma^\sigma$ is a convex open set, and that the level set $f^{-1}(\sigma)=\de\Gamma^\sigma$ is a smooth hypersurface in $\mathbb{R}^n$ which doesn't intersect the boundary of $\Gamma$.

\bigskip

We now recall the definition of $\mathcal{C}$-subsolutions for non-degenerate equations, as given by Sz\'ekelyhidi \cite[Definition 1]{Szekelyhidi18}.
\begin{definition}\label{subsolution def 1}
A function $\underline{u}$ on $X$ is said to be a $\mathcal{C}$-subsolution of the non-degenerate equation $F(A)=h$, if at each $x\in X$, the set
\[
(\mu(x)+\Gamma_n)\cap\de\Gamma^{h(x)}
\]
is bounded, where
\[
\mu(x) = \lambda\big(g^{j\ov{p}}\underline{\chi}_{i\ov{p}}(x)\big), \quad \underline{\chi}=\chi_0+\ddbar\underline{u}.
\]
\end{definition}

In \cite{Szekelyhidi18}, Sz\'ekelyhidi also proved the following equivalent characterization of $\mathcal{C}$-subsolutions. Following Trudinger \cite{Trudinger95}, define
\[
\ti{\Gamma} = \{ \mu\in\mathbb{R}^n ~|~ \text{there exists $t>0$ such that $\mu+t\mathbf{e}_i\in\Gamma$ for all $i$}\}.
\]
Note that $\ti{\Gamma}$ is also an open, symmetric, convex cone. For each $\mu\in\ti{\Gamma}$, we define
\[
f_{\infty, i}(\mu) = \lim_{t\rightarrow\infty}f(\mu+t\mathbf{e}_i), \quad
f_{\infty,\min}(\mu) = \min_{1\leq i\leq n}f_{\infty,i}(\mu),
\]
where here $\mathbf{e}_i\in\R^n$ is the $i^{th}$ standard basis vector. Note that, from the concavity of $f$, we have that $f_{\infty,\min}$ is concave and is thus either everywhere finite on $\ti{\Gamma}$, or uniformly $+\infty$ (cf. \cite{Trudinger95}).

\begin{proposition}\label{ASDFS}
(Sz\'ekelyhidi, \cite[p.345]{Szekelyhidi18}) The smooth function $\underline{u}$ is a $\mathcal{C}$-subsolution to the non-degenerate equation $F(A) = h$ if and only if:
\[
\mu(x) \in \ti{\Gamma}\ \ \text{and}\ \ f_{\infty,\min}(\mu) > h(x),\ \ \text{for any $x\in X$}.
\]
\end{proposition}
\begin{proof}
We first check the forward implication. Suppose that $x\in X$, $t > 0$, and $i\in\{1,\ldots, n\}$ are such that $\mu(x) + t\mathbf{e}_i\not\in \Gamma$. Since $\Gamma$ is an open convex cone containing $\Gamma_n$, there must exist some $s_0 > 0$ such that $\mu(x) + t\mathbf{e}_i + s_0\mathbf{1}\in\p \Gamma$. Further, our assumptions on $f$ imply that:
\[
\lim_{s\rightarrow\infty} f(\mu(x) + t\mathbf{e}_i + s\mathbf{1}) = \sup_\Gamma f.
\]
Since $\sup_{\p\Gamma} f < h < \sup_\Gamma f$, it follows that there exists some $s_1 > s_0 > 0$ with:
\[
f(\mu(x) + t\mathbf{e}_i + s_1\mathbf{1}) = h(x).
\]
Thus:
\[
\mu(x) + t\mathbf{e}_i + s_1\mathbf{1} \in (\mu(x) + \Gamma_n)\cap \p\Gamma^{h(x)}.
\]
So if $\underline{u}$ is a $\mathcal{C}$-subsolution, it follows that, for any $i$, the set of all $t$ such that $\mu(x) + t\mathbf{e}_i\not\in\Gamma$ is bounded; which is to say, $\mu(x)\in\ti{\Gamma}$.

We check the second requirement in a similar manner -- let $\mu(x)\in \ti{\Gamma}$ be such that:
\[
f_{\infty, i}(\mu(x)) = \lim_{t\rightarrow\infty} f(\mu(x) + t\mathbf{e}_i) \leq h(x),
\]
for some $i\in\{1,\ldots, n\}$. Since $f$ is strictly increasing in the $\mathbf{e}_i$-direction, we have:
\[
f(\mu(x) + t\mathbf{e}_i) < h(x) \text{ for all $t$ sufficiently large.}
\]
But as we just argued, this implies there is an $s_1 > 0$ such that:
\[
\mu(x) + t\mathbf{e}_i + s_1\mathbf{1} \in (\mu(x) + \Gamma_n)\cap\p\Gamma^{h(x)}
\]
for each $t$ sufficiently large. Clearly,
\[
\lim_{t\rightarrow\infty}|\mu(x)+t\mathbf{e}_{i}+s_{1}\mathbf{1}|
\geq \lim_{t\rightarrow\infty}\left(t-|\mu(x)|\right) = +\infty.
\]
Thus, if $\underline{u}$ is a $\mathcal{C}$-subsolution, then we cannot have $f_{\infty, i}(\mu(x)) \leq h(x)$ for any $i$. This concludes the forward direction.

Conversely, suppose that $\mu(x)\in\ti{\Gamma}$ and $f_{\infty,\min}(\mu(x)) > h(x)$ but $(\mu(x) + \Gamma_n)\cap\p\Gamma^{h(x)}$ is unbounded. Then there exists an unbounded sequence $\{v_j\}\in (\mu(x) + \Gamma_n)\cap\p\Gamma^{h(x)}$ -- if we write $v_j = (v_j^1,\ldots, v_j^n)$, then after taking a subsequence, we can assume that:
\[
\lim_{j\rightarrow\infty} v_j^{i_0} = +\infty,
\]
for at least one index $i_0$. But then:
\[
f_{\infty,\min}(\mu(x)) \leq f_{\infty,i_0}(\mu(x))\leq \lim_{j\rightarrow\infty} f(\mu(x) + v_j) = h(x),
\]
contradicting with $f_{\infty,\min}(\mu(x))>h(x)$.
\end{proof}

We will take this alternate characterization as our definition of a $\mathcal{C}$-subsolution for degenerate equations:
\begin{definition}\label{subsolution def 2}
We say that a function $\underline{u}$ on $X$ is a $\mathcal{C}$-subsolution to the degenerate equation \eqref{degen equation} if at each $x\in X$, we have:
\[
\mu(x) \in \ti{\Gamma}, \quad f_{\infty,\min}(\mu) > h(x).
\]
\end{definition}

The advantage of using this characterization of $\mathcal{C}$-subsolutions for degenerate equations is that it is easily seen to be preserved under small perturbations of $h$. As mentioned above, $f_{\infty,\min}$ is either $\equiv +\infty$, or everywhere finite and concave on the open cone $\ti{\Gamma}$; it follows that, in this second case, $f_{\infty,\min}(\mu(x))$ is continuous on $X$. Thus, we immediately get:

\begin{proposition}\label{sigma 0}
If $\underline{u}$ is a $\mathcal{C}$-subsolution to the non-degenerate equation \eqref{new equation} or the degenerate equation \eqref{degen equation}, then there exists a constant $\sigma_0 > 0$, depending only on $(X,\omega)$, $\Gamma$, $f$, $\chi_{0}$, $\underline{u}$ and $h$, such that $\underline{u}$ is a $\mathcal{C}$-subsolution to the non-degenerate equation:
\[
F(A) = h + 2\sigma_0.
\]
Consequently, if $\ti{h}\in C^\infty(X)$ satisfies $\sup_{\p\Gamma} f \leq \ti{h} <\sup_\Gamma f$ and $|\ti{h} - h| < 2\sigma_0$, then $\underline{u}$ is also a $\mathcal{C}$-subsolution to $F(A) = \ti{h}$.
\end{proposition}

Given this, we will often refer to a $\mathcal{C}$-subsolution to \eqref{degen equation} as a $\mathcal{C}$-subsolution to \eqref{new equation}, leaving context to dictate if the equation is degenerate or not.

\begin{remark}\label{sigma eq 1}
For many geometric PDEs (including the complex Monge-Amp\`ere equation, the complex Hessian equation and the Monge-Amp\`ere equation for $(n-1)$-plurisubharmonic functions), it is easy to see that $\sup_\Gamma f = +\infty$ and $f_{\infty,\min} \equiv +\infty$. This implies that one can always choose $\sigma_0 = 1$ in this case, i.e. it can be taken to be a universal constant independent of $\underline{u}$ and $h$.
\end{remark}

\section{Estimates for the Complex Hessian}\label{section C subsolution}

In this section, we will show that the {\em a priori} estimates of \cite{Szekelyhidi18} hold in the degenerate setting. Specifically, we claim:

\begin{theorem}\label{Szekelyhidi estimates}
Let $u$ be a smooth solution of \eqref{main eqn} with $\sup_X u=-1$. Suppose that $\underline{u}$ is a $\mathcal{C}$-subsolution of \eqref{main eqn}, then there exists a constant $C$ depending only on $(X,\omega)$, $\Gamma$, $f$, $\chi_0$, $\underline{u}$, $\sigma_0$, $\sup_{X}h$, $\sup_{X}|\de h|_g$ and on a lower bound of $\ddbar h$ such that
\[
\sup_{X}|u|+\sup_{X}|\de u|_g+\sup_{X}|\de\dbar u|_g \leq C.
\]
\end{theorem}

In \cite{Szekelyhidi18}, Sz\'ekelyhidi established that
\[
\sup_X |u|\leq C \text{ and } \sup_X|\p\pbar u|_g\leq C\sup_X|\p u|^2_g + C
\]
for some constant $C$ depending only on $(X,\omega)$, $\Gamma$, $f$, $\chi_0$, $\underline{u}$, $\sup_{X}h$, $\sup_{X}|\de h|_g$, a lower bound of $\ddbar h$, and on four geometric constants, $\kappa$, $\tau$, $R$, $\delta$, which might depend on $\inf_X h$. He also establishes the $C^1$ estimate using a blow-up argument, which can be seen to depend only on $C$, $(X,\omega)$, $\Gamma$, $f$, $\chi_0$, and $\sup_X h$. Thus, to prove Theorem \ref{Szekelyhidi estimates}, it will be enough to check that the constants $\kappa, \tau, R,$ and $\delta$ can be made independent of $\inf_X h$.

\subsection{The constants $\kappa$ and $\tau$:}

We begin by showing that, if one is more careful with the geometry of the $\mathcal{C}$-subsolution, the constant $\kappa$ in \cite[Proposition 5]{Szekelyhidi18} (which is a refinement of \cite[Theorem 2.16]{Guan14}) can be made quite explicit:

\begin{lemma}
Suppose that $\mu \in \ti{\Gamma}$, $\sigma\in(\sup_{\de\Gamma}f,\sup_{\Gamma}f)$ are such that:
$${\sigma + \sigma_0} < \sup_\Gamma f,\ \ \ \ \mu - 2\delta\mathbf{1}\in\ti{\Gamma},$$
and
\begin{equation}\label{bdd}
(\mu-2\delta\mathbf{1}+\Gamma_n)\cap\de\Gamma^{{\sigma + \sigma_0}} \subseteq B_R(0),
\end{equation}
for some constants $\delta, R, \sigma_0 > 0$. Then there exists a constant $\kappa > 0$, depending only on $\mu$, $\delta$, $R$, $n$ and $\Gamma$ such that, for any $\lambda\in \p\Gamma^{\sigma}$, we either have:
\begin{enumerate}
\setlength{\itemsep}{1mm}
\item $\sum_{i}f_i(\lambda)(\mu_i-\lambda_i)>\kappa\mathcal{F}(\lambda)$, or
\item $f_i(\lambda)>\kappa\mathcal{F}(\lambda)$ for all $i = 1,\ldots, n$,
\end{enumerate}
where here $f_i := \frac{\p f}{\p \lambda_i}$ and $\mathcal{F}(\lambda)=\sum_if_i(\lambda)$.
\end{lemma}

\begin{proof}
We start with a claim:

{\bf Claim:} Suppose that $\lambda'\in (\mu-2\delta\mathbf{1}+\ov{\Gamma_n})\cap (\R^n\setminus \ov{\Gamma^{\sigma + \sigma_0}})$. Then for any non-zero $v\in\ov{\Gamma_n}$, there exists a $t > 0$ such that $\lambda' + tv\in\Gamma$ and $f(\lambda' + tv) \geq {\sigma + \sigma_0}$.

{\bf Proof of Claim:} That $\lambda' + tv\in\Gamma$ for any non-zero $v\in\ov{\Gamma_n}$ and $t$ sufficiently large follows immediately from the fact that $\mu-2\delta\mathbf{1}\in\ti{\Gamma}$.

Suppose then for the sake of a contradiction that the other conclusion does not hold, i.e. that there exists a non-zero $v\in\ov{\Gamma_n}$ such that $f(\lambda' + tv) < {\sigma + \sigma_0}$ for all $t > 0$ sufficiently large. By property (ii) of $f$, for each $t$ there exists a $w_t\in\Gamma_n$ such that $f(\lambda' + tv + w_t) = {\sigma + \sigma_0}$, i.e. $\lambda' + tv + w_t \in \p\Gamma^{{\sigma + \sigma_0}}$. Since $tv + w_t \in \Gamma_n$, it follows that:
\[
\lambda ' + tv + w_t \in (\mu - 2\delta\mathbf{1}+\Gamma_n)\cap\p\Gamma^{{\sigma + \sigma_0}} \ \ \text{for each $t$}.
\]
By \eqref{bdd}, the above set is bounded though, so we have:
\[
|\lambda ' + tv + w_t| \leq R \ \ \text{for each $t$},
\]
which is clearly absurd, as $|tv + w_t|\rightarrow\infty$ as $t\rightarrow\infty$.

\bigskip

We now begin the proof proper. Consider the set:
\[
A_\delta := (\mu - \delta\mathbf{1} + \ov{\Gamma_n})\cap(\R^n\setminus\ov{\Gamma^{\sigma + \sigma_0}}).
\]
Suppose for now that $A_\delta$ is non-empty -- we will deal with the case $A_{\delta}=\emptyset$ at the end of the proof. The assumption that $A_\delta\not=\emptyset$ is clearly equivalent to assuming that $\mu-\delta\mathbf{1}\in\R^n\setminus\ov{\Gamma^{\sigma + \sigma_0}}$.

For each $v\in A_\delta$, we define the associated convex cone:
\[
\mathcal{C}_v := \{ w\in \R^n\ |\ v + tw\in (\mu - 2\delta\mathbf{1}+\ov{\Gamma_n})\cap \ov{\Gamma^{{\sigma + \sigma_0}}}\text{ for some }t>0\}.
\]
Applying the above claim twice (once with $\lambda'=\mu-\delta\mathbf{1}$ and once with $\lambda'=\mu-2\delta\mathbf{1}$), we see that
\begin{itemize}
\item $\ov{\Gamma_n}\setminus\{0\}\subset \mathcal{C}_{\mu- \delta\mathbf{1}}$, and
\item for each index $i=1,\ldots, n$, set:
\[
T_i := \{t  > 0\ |\ \mu - 2\delta\mathbf{1} + t\mathbf{e}_i \in \ov{\Gamma^{\sigma + \sigma_0}}\}
\]
is non-empty. Set $t_i := \min T_i$.
\end{itemize}

Let $\mathbf{1}_{\hat{i}} := \sum_{j\not=i}\mathbf{e}_j$. The second point implies that we can actually find a larger cone than $\ov{\Gamma_n}\setminus\{0\}$ inside $\mathcal{C}_{\mu-\delta\mathbf{1}}$ by rewriting:
\[
\mu - \delta\mathbf{1} + (t_i - \delta)\mathbf{e}_i - \delta\mathbf{1}_{\hat{i}} = \mu - 2\delta\mathbf{1} + t_i\mathbf{e}_i \in (\mu - 2\delta\mathbf{1}+\ov{\Gamma_n})\cap \ov{\Gamma^{{\sigma + \sigma_0}}}.
\]
By definition, we thus get:
\[
(t_i - \delta)\mathbf{e}_i - \delta\mathbf{1}_{\hat{i}} \in \mathcal{C}_{\mu - \delta\mathbf{1}}.
\]

It follows now that if $w = (w_1,\ldots, w_n) \in\mathcal{C}_{\mu - \delta\mathbf{1}}^*$ (the dual cone of $\mathcal{C}_{\mu-\delta\mathbf{1}}$) then we must have:
\[
\langle w, (t_i - \delta)\mathbf{e}_i - \delta\mathbf{1}_{\hat{i}}\rangle \geq 0.
\]
Since $\Gamma_n\subset\mathcal{C}_{\mu-\delta\mathbf{1}}$, then $\mathcal{C}_{\mu-\delta\mathbf{1}}^{*}\subset\Gamma_{n}$. Each component of $w$ must be positive, so we have:
\[
t_i w_i \geq (t_i - \delta) w_i \geq \delta \sum_{j\not = i} w_j.
\]
Squaring both sides and dropping cross-terms gives:
\[
t_i^2 w_i^2 \geq \delta^2\left(\sum_{j\not= i} w_j\right)^2 \geq \delta^2\sum_{j\not=i} w_j^2 = \delta^2 (|w|^2 - w_i^2),
\]
implying
\[
w_i^2 \geq \frac{\delta^2}{t_i^2 + \delta^{2}}|w|^2.
\]

It is clear that $\mu - 2\delta\mathbf{1} + t_i\mathbf{e}_i \in \p\Gamma^{\sigma + \sigma_0}$. By \eqref{bdd}, $|\mu - 2\delta\mathbf{1} + t_i\mathbf{e}_i| \leq R$, so:
\[
t_i \leq R + |\mu - 2\delta\mathbf{1}| \leq C(\mu, \delta, R, n).
\]
It thus follows that:
\[
w_i \geq \frac{1}{C} |w|
\]
for any vector $w\in\mathcal{C}_{\mu - \delta\mathbf{1}}^*$ and some constant $C$ independent of $\sigma$ and $\sigma_0$.

\bigskip

We now show the same result holds for any $v\in A_\delta$, i.e.,
\[
w_i \geq \frac{1}{C} |w| \ \ \text{for any $w\in\mathcal{C}_{v}^*$}.
\]
This will follow immediately if we can show:
\[
\mathcal{C}_{\mu-\delta\mathbf{1}} \subseteq \mathcal{C}_{v}.
\]
This is quite clear however-- suppose that $w\in \mathcal{C}_{\mu - \delta\mathbf{1}}$, so that there exists $t > 0$ with:
\[
\mu - \delta\mathbf{1} + tw \in (\mu - 2\delta\mathbf{1}+\ov{\Gamma_n})\cap \ov{\Gamma^{{\sigma + \sigma_0}}}.
\]
Since $v\in A_\delta$, then there exists a $w_2\in\ov{\Gamma_n}$ with $\mu - \delta\mathbf{1} + w_2 = v$. It follows then that:
\[
{\sigma + \sigma_0} \leq f(\mu - \delta\mathbf{1} + tw) \leq f(v + tw),
\]
so $w\in\mathcal{C}_{v}$ by definition.

\bigskip

We may conclude as follows (we now drop the assumption that $A_\delta$ is non-empty). Let $\lambda\in\de\Gamma^{\sigma}$ and define $T_\lambda$ to be the tangent plane to $\de\Gamma^{\sigma}$ at $\lambda$ with inward normal vector
\[
\mathbf{n}_\lambda = \left(f_1(\lambda),\dots,f_n(\lambda)\right).
\]
Define $H_\lambda$ to be the upper-half space:
\[
H_\lambda = \{w\in\mathbb{R}^n~|~(w-\lambda)\cdot\mathbf{n}_\lambda\geq0\}.
\]
Note that $\ov{\Gamma^{\sigma+\sigma_0}}\subseteq \ov{\Gamma^{\sigma}} \subseteq H_\lambda$, by concavity of $f$. We now have two cases:

\bigskip
\noindent
{\bf Case 1:} $T_\lambda\cap A_\delta\neq\emptyset$.

\bigskip

In this case, choose some $v\in T_\lambda\cap A_\delta$. By definition, the cone $v + \mathcal{C}_v$ has cross section $(\mu-2\delta\mathbf{1}+\ov{\Gamma_n})\cap\de\Gamma^{{\sigma + \sigma_0}}$, so clearly
\[
v+\mathcal{C}_v\subset H_\lambda.
\]
It follows then that $\mathbf{n}_\lambda\in\mathcal{C}_v^*$ (as $(v - \lambda)\cdot\mathbf{n}_\lambda = 0$), so we get
\[
f_i(\lambda) \geq \frac{1}{C}|\nabla f|.
\]
Combining this with $|\nabla f| \geq \frac{1}{\sqrt{n}}\mathcal{F} > 0$ (cf. \cite[(19)]{Szekelyhidi18}), we see that (2) holds.

\bigskip
\noindent
{\bf Case 2:} $T_\lambda\cap A_\delta=\emptyset$.

\bigskip

There are two possibilities: either $A_\delta=\emptyset$ or $A_\delta\not=\emptyset$ and $T_\lambda\cap A_\delta = \emptyset$.

\noindent As mentioned previously, $A_\delta = \emptyset$ if and only if $\mu-\delta\mathbf{1}\in\ov{\Gamma^{\sigma + \sigma_0}}\subseteq H_\lambda$. If alternatively $A_\delta\neq\emptyset$, then $A_\delta$ lies on one side of $T_\lambda$, as $A_\delta\cap T_\lambda = \emptyset$. Thanks to the claim, $A_\delta\cap\ov{\Gamma^{\sigma}}\not=\emptyset$, so we must have $A_\delta\subseteq H_\lambda$.

Thus, in both cases we have $\mu-\delta\mathbf{1}\in H_\lambda$. It then follows that
\[
(\mu-\delta\mathbf{1}-\lambda)\cdot\mathbf{n}_\lambda \geq 0,
\]
which implies
\[
\sum_{i} f_i(\lambda)\,(\mu_i - \lambda_i) \geq \delta\mathcal{F}(\lambda),
\]
finishing the proof.
\end{proof}

The next lemma is essentially the same as \cite[Lemma 9]{Szekelyhidi18}.

\begin{lemma}
For any $\sup_{\de\Gamma}f<\sigma\leq\Lambda<\sup_{\Gamma}f$, there exists a constant $\tau > 0$ depending only on $\Lambda$, $f$ and $\Gamma$ such that
\[
\mathcal{F}(\lambda) > \tau \quad \text{for any $\lambda\in\de\Gamma^\sigma$}.
\]
\end{lemma}

\begin{proof}
Take $\sigma'\in(\Lambda,\sup_{\Gamma}f)$. By \cite[Lemma 9]{Szekelyhidi18}, there exists a constant $N$ depending only on $\Lambda$, $f$, and $\Gamma$ such that
\[
\Gamma+N\mathbf{1} \subset \Gamma^{\sigma'}.
\]
Using the concavity of $f$, for any $\lambda\in\de\Gamma^\sigma$,
\[
\sigma' < f(\lambda+N\mathbf{1}) \leq f(\lambda)+N\sum_{i}f_i(\lambda) = \sigma+N\mathcal{F}(\lambda).
\]
This implies
\[
\mathcal{F}(\lambda) \geq \frac{\sigma'-\sigma}{N} \geq \frac{\sigma'-\Lambda}{N}.
\]
\end{proof}

\subsection{The constants $\delta$ and $R$}

We now show the existence of appropriate $\delta$ and $R$:

\begin{lemma}
Let $h$ be a smooth function on $X$ such that $\sup_{\de\Gamma}f< h\leq\Lambda<\sup_{\Gamma}f$ and $\sup_{X}|\de h|_g\leq\Lambda$. Suppose that $\underline{u}$ is a $\mathcal{C}$-subsolution of $F(A)=h+2\sigma_0$. Then there exist constants $\delta$ and $R$, depending only on $(X,\omega)$, $\Gamma$, $f$, $\chi_0$, $\underline{u}$, $\sigma_0$, and $\Lambda$, such that for any $x\in X$,
\[
(\mu(x)-2\delta\mathbf{1}+\Gamma_n)\cap\de\Gamma^{h(x)+\sigma_0} \subset B_R(0).
\]
\end{lemma}

\begin{proof}
We argue by contradiction. Suppose that there exist sequences $\{h_i\}\in C^\infty(X)$ and $\{x_i\}\in X$ such that:
\begin{enumerate}[(a)]
\item $\sup_{\de\Gamma}f<h_i\leq\Lambda$ and $\sup_{X}|\de h_i|_g\leq\Lambda$,
\item $\underline{u}$ is a $\mathcal{C}$-subsolution of $F(A)=h_i+2\sigma_0$, and
\item there exists $v_i = \sum_{j=1}^n v_i^j\mathbf{e}_j\in\Gamma_n$ with $|v_i|>i$ such that $\mu(x_i)-2i^{-1}\mathbf{1}+v_i\in\Gamma$ and
\[
f(\mu(x_i)-2i^{-1}\mathbf{1}+v_i) = h_i(x_i)+\sigma_0.
\]
\end{enumerate}
After passing to a subsequence, we may assume that $x_i \rightarrow x_\infty\in X$. Fix an $\e > 0$ so that $\mu(x_\infty) - \e\mathbf{1}\in\ti{\Gamma}$ (which we can do since $\ti{\Gamma}$ is open). Then for all $i$ large enough, we have:
\[
\mu(x_i) - 2i^{-1}\mathbf{1} \in \mu(x_\infty) - \e\mathbf{1}+\Gamma_{n}.
\]
Since $v_i\in\Gamma_n$ and $|v_i| > i$, after passing to an additional subsequence, we can assume that $v_i^j \rightarrow\infty$ for some fixed index $j\in\{1,\ldots, n\}$. It follows then that:
\[
\begin{split}
f_{\infty, \min}&(\mu(x_\infty) - \e\mathbf{1}) \leq f_{\infty, j}(\mu(x_\infty) - \e\mathbf{1}) = \lim_{i\rightarrow\infty} f(\mu(x_\infty) - \e\mathbf{1} + v_i^j\mathbf{e}_j)\\
& \leq \lim_{i\rightarrow\infty} f(\mu(x_i) - 2i^{-1}\mathbf{1} + v_i) = \lim_{i\rightarrow\infty} h_i(x_i) + \sigma_0.
\end{split}
\]

By the Arzel\`a-Ascoli theorem, after passing to a subsequence once more, we can assume that the $h_i$ either converge uniformly to some $h_\infty\in C(X)$ or that $\sup_X h_i\searrow-\infty$, as the $h_i$ are uniformly bounded above by $\Lambda$. In the second case, we conclude that $f_{\infty, \min}(\mu(x_\infty) - \e\mathbf{1}) = -\infty$, which violates previously stated properties of $f_{\infty,\min}$ (see the paragraph above Proposition \ref{ASDFS}). Thus, we must have that $h_i \xrightarrow{C_0} h_\infty\in C(X)$, so that:
\[
f_{\infty,\min} (\mu(x_\infty) - \e\mathbf{1}) \leq h_\infty(x_\infty) + \sigma_0.
\]
Letting $\e\rightarrow 0$ gives:
\[
f_{\infty,\min} (\mu(x_\infty)) \leq h_\infty(x_\infty) + \sigma_0.
\]
But since $\underline{u}$ is a $\mathcal{C}$-subsolution to $F(A) = h_i + 2\sigma_0$ for each $i$, we must have:
\[
f_{\infty,\min}(\mu(x_\infty)) > h_i(x_\infty) + 2\sigma_0.
\]
Letting $i\rightarrow\infty$ gives
\[
f_{\infty,\min}(\mu(x_\infty) \geq h_\infty(x_\infty) + 2\sigma_0,
\]
which is a contradiction, as desired.
\end{proof}

Substituting the above lemmas into the proof of \cite[Theorem 2.18]{Guan14} (see also \cite[Proposition 6]{Szekelyhidi18}), we then obtain the following proposition:
\begin{proposition}\label{subsolution prop}
Let $u$ be a smooth solution of \eqref{main eqn} and suppose that $\underline{u}$ is a $\mathcal{C}$-subsolution of \eqref{main eqn}. Set
\[
\underline{\chi} = \chi_0+\ddbar\underline{u}.
\]
Then there exist positive constants $\kappa$ and $\tau$ depending only on $(X,\omega)$, $\Gamma$, $f$, $\chi_0$, $\underline{u}$, $\sigma_0$, $\sup_{X}h$, $\sup_{X}|\de h|_{g}$ such that at each $x\in X$, both:
\begin{enumerate}
\item either $\sum_{i}F^{i\ov{j}}(\underline{\chi}_{i\ov{j}}-\chi_{i\ov{j}})>\kappa\mathcal{F}$, or $F^{i\ov{i}}>\kappa\mathcal{F}$ for all $i$; and
\item $\mathcal{F}\geq\tau$,
\end{enumerate}
where $F^{i\ov{j}}$ is the first derivative of $F$ and $\mathcal{F}=\sum_{i}F^{i\ov{i}}$.
\end{proposition}


\section{Proof of Theorem \ref{main estimate}}

To prove Theorem \ref{main estimate}, it suffices to establish the following real Hessian estimate:
\begin{theorem}\label{real Hessian estimate}
Let $u$ be a smooth solution of \eqref{main eqn} with $\sup_Xu=-1$. Suppose that $\underline{u}$ is a $\mathcal{C}$-subsolution of \eqref{main eqn}, then there exists a constant $C$ depending only on $(X,\omega)$, $\sup_{X}|u|$, $\sup_{X}|\de u|_g$, $\sup_{X}|\de\dbar u|_g$, $\chi_0$, $\underline{u}$, $\sigma_0$, $\sup_{X}h$, $\sup_{X}|\de h|_g$ and on a lower bound of $\nabla^2 h$ such that
\[
\sup_{X}|\nabla^2 u|_g \leq C.
\]
\end{theorem}

First, we assume without loss of generality that $\underline{u}=0$. Otherwise, we replace $\chi_0$ by $\underline{\chi}$. Let $\lambda_1(\nabla^{2}u)\geq\lambda_2(\nabla^{2}u)\geq\cdots\geq\lambda_{2n}(\nabla^{2}u)$ be the eigenvalues of real Hessian $\nabla^{2}u$ with respect to the Riemannian metric $g$. Under the assumptions of $\Gamma$, we see that
\[
\Gamma \subset \Gamma_1 = \{(v_1,\cdots,v_n)\in\mathbb{R}^{n}~|~\sum_iv_i > 0 \}.
\]
As such, we have $g^{i\ov{j}}\chi_{i\ov{j}}>0$, which implies (cf. \eqref{Theta})
\[
\sum_{\alpha=1}^{2n}\lambda_{\alpha}(\nabla^{2}u) = \Delta u > -C.
\]
This shows
\[
|\nabla^{2}u|_g = \left(\sum_{\alpha=1}^{2n}\left(\lambda_{\alpha}(\nabla^{2}u)\right)^{2}\right)^{\frac{1}{2}} \leq C\lambda_1(\nabla^{2}u)+C.
\]
Hence, it suffices to show that $\lambda_1(\nabla^{2}u)$ is uniformly bounded from above.

We define
\[
L := \sup_{X}|\nabla^{2}u|_g+1, \quad \rho := \nabla^{2}u+Lg.
\]
Clearly, $\rho$ is positive everywhere and $|\rho|_{g}^{2}\leq4L^{2}$. We consider the following quantity
\[
Q = \log\lambda_1(\nabla^{2}u)+\xi(|\rho|_g^2)+\eta(|\de u|_g^2)+e^{-Au},
\]
where
\[
\xi(s) = -\frac{1}{3}\log\left(5L^{2}-s\right), \quad
\eta(s) = -\frac{1}{3}\log\left(1+\sup_{X}|\de u|_g^2-s\right)
\]
and $A\gg1$ is a large constant to be determined later. It is very easy to check:
\[
\xi'' = 3(\xi')^2, \quad
\eta'' = 3(\eta')^2.
\]

We will show that $Q$ is uniformly bounded above using the maximum principle. We first show that this is sufficient to conclude that $\lambda_1(\nabla^2 u)$ is also uniformly bounded. Suppose then that $Q\leq C$,  and let $y\in X$ be a point at which $\lambda_1(\nabla^2 u)$ attains its supremum. It follows from the above that:
\[
\lambda_1(\nabla^2 u)(y) = \sup_X \lambda_1(\nabla^2 u) \geq \frac{1}{C_1}L - C_2,
\]
for uniform constants $C_1, C_2$. Then, since $|u|$ and $|\de u|^2_g$ are uniformly bounded, at $y$, we have:
\[
C \geq \log\lambda_1 - \frac{1}{3}\log(5L^2 - |\rho|^2_g) \geq \log\left(L - C_1 C_2\right) - \frac{1}{3}\log(L^2) - \log C,
\]
as $|\rho|^2_g \geq 0$. This shows
\[
C \geq \frac{1}{3}\log L - \log C
\]
so that $L \leq C$. By definition, $\sup_X \lambda_1(\nabla^2 u) \leq L$, so the theorem will follow.

\bigskip

We thus bound $Q$ from above. Let $x_0$ be a maximum point of $Q$. We choose real coordinates $\{x^\alpha\}_{\alpha=1}^{2n}$ near $x_0$ such that at $x_{0}$,
\[
g_{\alpha\beta} := g\left(\frac{\de}{\de x^{\alpha}},\frac{\de}{\de x^{\beta}}\right) = \delta_{\alpha\beta}, \quad
\frac{\de g_{\alpha\beta}}{\de x^{\gamma}} = 0, \quad \text{for any $\alpha,\beta,\gamma=1,2,\ldots,2n$}.
\]
and
\[
J\left(\frac{\de}{\de x^{2i-1}}\right) = \frac{\de}{\de x^{2i}}, \quad \text{for any $i=1,2,\ldots,n$},
\]
where $J$ denotes the complex structure of $(X,\omega)$. Define at $x_{0}$
\[
e_{i} = \frac{1}{\sqrt{2}}\left(\frac{\de}{\de x^{2i-1}}-\sqrt{-1}\frac{\de}{\de x^{2i}}\right), \quad \text{for any $i=1,2,\ldots,n$}.
\]
Thus $g_{i\ov{j}}:=g(e_{i}, \ov{e}_{j})=\delta_{ij}$ and $\{e_{i}\}_{i=1}^{n}$ forms a frame of $(1,0)$ vectors at $x_0$. After rotating coordinates, we further assume that
\[
\chi_{i\ov{j}} = \chi_{i\ov{i}}\delta_{ij}, \quad
\chi_{1\ov{1}} \geq \chi_{2\ov{2}}\geq \cdots \geq \chi_{n\ov{n}}, \quad \text{at $x_0$}.
\]
We then extend $\{e_{i}\}_{i=1}^{n}$ smoothly to a $g$-unitary frame of $(1,0)$ vectors near $x_0$.

At $x_0$, let $V_{\alpha}$ be the $g$-unit eigenvector of $\lambda_{\alpha}$ and denote its components by $\{V_{\alpha}^{\beta}\}_{\beta=1}^{2n}$. Extend $V_1, V_2,\ldots,V_{2n}$ to be vector fields in a neighborhood of $x_0$ by taking the components to be constant.

When $\lambda_1(\nabla^{2} u)(x_0)=\lambda_2(\nabla^{2} u)(x_0)$, the quantity $Q$ might not be smooth at $x_0$. To avoid this situation, we apply a standard perturbation argument (see e.g., \cite{Szekelyhidi18,CTW19}). Define
\[
B = B_{\alpha\beta}dx^{\alpha}\otimes dx^{\beta} := (\delta_{\alpha \beta}-V_{1}^{\alpha}V_{1}^{\beta})dx^{\alpha}\otimes dx^{\beta}
\]
and
\[
\Phi = \Phi_{\beta}^{\alpha}\frac{\de}{\de x^{\alpha}}\otimes dx^{\beta} := g^{\alpha\gamma}(\nabla_{\gamma}\nabla_{\beta}u-B_{\gamma\beta})\frac{\de}{\de x^{\alpha}}\otimes dx^{\beta},
\]
where
\[
\nabla_{\gamma}\nabla_{\beta}u =
\nabla^{2}u\left(\frac{\de}{\de x^{\gamma}},\frac{\de}{\de x^{\beta}}\right).
\]
Write $\lambda_{\alpha}(\Phi)$ for the eigenvalues of $\Phi$ with respect to $g$. It is clear that $\lambda_{1}(\Phi)(x_0)>\lambda_{2}(\Phi)(x_0)$, which implies $\lambda_{1}(\Phi)$ is smooth at $x_{0}$. So we consider the modified quantity near $x_{0}$:
\[
\hat{Q} = \log\lambda_1(\Phi)+\xi(|\rho|_g^2)+\eta(|\de u|_g^2)+e^{-Au}.
\]
Since $\lambda_{1}(\Phi)(x_{0})=\lambda_{1}(\nabla^{2}u)(x_{0})$ and $\lambda_{1}(\nabla^{2}u)\geq \lambda_{1}(\Phi)$ near $x_{0}$, we see that $\hat{Q}$ still achieves its maximum at $x_{0}$. Note that the vector fields $V_1,V_2\ldots,V_{2n}$ are still eigenvectors for $\Phi$ at $x_0$, with eigenvalues $\lambda_1(\Phi)>\lambda_2(\Phi)\geq\ldots\geq\lambda_{2n}(\Phi)$. In the following argument, we use $\lambda_{\alpha}$ to denote $\lambda_{\alpha}(\Phi)$ for convenience.

Let $F^{i\ov{j}}$ and $F^{i\ov{j},k\ov{l}}$ be the first and second derivatives of $F$. Then at $x_0$, we have (see e.g. \cite{Andrews94,Gerhardt96,Spruck05})
\[
F^{i\ov{j}} = \delta_{ij}f_i
\]
and
\begin{equation}\label{concave}
F^{i\ov{j},p\ov{q}} = f_{ip}\delta_{ij}\delta_{pq}+\frac{f_i-f_j}{\chi_{i\ov{i}}-\chi_{j\ov{j}}}(1-\delta_{ij})\delta_{iq}\delta_{jp}.
\end{equation}
Note that the second term is interpreted as a limit if $\chi_{i\ov{i}}=\chi_{j\ov{j}}$. Recalling that $\chi_{1\ov{1}} \geq \chi_{2\ov{2}}\geq \cdots \geq \chi_{n\ov{n}}$, we have (see e.g. \cite{EH89,Spruck05})
\[
F^{1\ov{1}} \leq F^{2\ov{2}} \leq \cdots \leq F^{n\ov{n}}.
\]

We always use subscripts to denote covariant differentiation with respect to the Levi-Civita connection. Since $(X,\omega)$ may be not K\"ahler, the Levi-Civita connection will in general not be compatible with the complex structure. Thus, there exists a tensor field $\Theta$ depending only on $(X,\omega)$ such that (cf. \cite[Section 2]{Chu18})
\begin{equation}\label{Theta}
(\de\dbar u)(e_i,\ov{e}_{j}) = u_{i\ov{j}}+\Theta_{i\ov{j}}^ku_k+\Theta_{i\ov{j}}^{\ov{k}}u_{\ov{k}}
\end{equation}
with
\[
\ov{\Theta_{i\ov{j}}^k} = \Theta_{j\ov{i}}^{\ov{k}} \text{ and } \ov{\Theta_{i\ov{j}}^{\ov{k}}} = \Theta_{j\ov{i}}^{k}.
\]
We emphasize that $u_{i\ov{j}}$ represents $(\nabla^2u)(e_i,\ov{e}_{j})$, not $(\de\dbar u)(e_i,\ov{e}_{j})$. In particular,
\begin{equation}\label{Theta 1}
(\de\dbar u)(e_i,\ov{e}_{i}) = u_{i\ov{i}}+2\mathrm{Re}\big(\Theta_{i\ov{i}}^ku_k\big).
\end{equation}
The following commutation formula for covariant derivatives will be used many times:
\[
(\nabla^{p}u)_{\alpha\beta}-(\nabla^{p}u)_{\beta\alpha}
= (\nabla^{p}u)*\mathrm{Rm},
\]
where $*$ means contraction and $\mathrm{Rm}$ denotes the curvature tensor of $(X,g)$.

\subsection{Some calculations}
By the maximum principle, at $x_{0}$, we have
\begin{equation}\label{Q ii}
\begin{split}
0 \geq {} & F^{i\ov{i}}\hat{Q}_{i\ov{i}} \\
= {} & \frac{F^{i\ov{i}}(\lambda_{1})_{i\ov{i}}}{\lambda_{1}}-\frac{F^{i\ov{i}}|(\lambda_1)_i|^{2}}{\lambda_{1}^{2}}
+\xi'F^{i\ov{i}}(|\rho|_g^2)_{i\ov{i}}+\xi''F^{i\ov{i}}|(|\rho|_g^2)_i|^{2} \\
& +\eta'F^{i\ov{i}}(|\de u|_g^2)_{i\ov{i}}+\eta''F^{i\ov{i}}|(|\de u|_g^2)_i|^{2} \\[2mm]
& +A^2e^{-Au}F^{i\ov{i}}|u_i|^2-Ae^{-Au}F^{i\ov{i}}u_{i\ov{i}}.
\end{split}
\end{equation}
The task of this subsection is to obtain a lower bound of $F^{i\ov{i}}\hat{Q}_{i\ov{i}}$.
\begin{lemma}\label{lem 1}
At $x_0$, we have
\begin{equation}\label{lem 1 eqn 1}
F^{i\ov{i}}u_{i\ov{i}} = -F^{i\ov{i}}(\underline{\chi}_{i\ov{i}}-\chi_{i\ov{i}})
-2F^{i\ov{i}}\mathrm{Re}\big(\Theta_{i\ov{i}}^ku_k\big),
\end{equation}

\begin{equation}\label{lem 1 eqn 2}
\begin{split}
F^{i\ov{i}}(|\de u|_g^2)_{i\ov{i}} \geq {} & \sum_{p}F^{i\ov{i}}(|u_{ip}|^2+|u_{i\ov{p}}|^2)-C\mathcal{F} \\
& -2\sum_{p}F^{i\ov{i}}\mathrm{Re}\left(\Theta_{i\ov{i}}^ku_{k\ov{p}}u_p+\ov{\Theta_{i\ov{i}}^k}u_{\ov{k}\ov{p}}u_p\right),
\end{split}
\end{equation}

\begin{equation}\label{lem 1 eqn 3}
\begin{split}
F^{i\ov{i}}(\lambda_1)_{i\ov{i}} \geq {} &
2\sum_{\alpha>1}\frac{F^{i\ov{i}}|u_{V_{1}V_{\alpha}i}|^2}{\lambda_1-\lambda_\alpha}
-F^{i\ov{j},p\ov{q}}\chi_{i\ov{j}V_{1}}\chi_{p\ov{q}V_{1}} \\
&-2F^{i\ov{i}}\mathrm{Re}\big(\Theta_{i\ov{i}}^ku_{V_1V_1k}\big)-C\lambda_1\mathcal{F},
\end{split}
\end{equation}

\begin{equation}\label{lem 1 eqn 4}
F^{i\ov{i}}(|\rho|_{g}^{2})_{i\ov{i}}
\geq 2\sum_{\alpha,\beta}F^{i\ov{i}}|u_{i\alpha\beta}|^{2}
-4\sum_{\alpha,\beta}F^{i\ov{i}}\rho_{\alpha\beta}\mathrm{Re}(\Theta_{i\ov{i}}^ku_{\alpha\beta k})-CL^2\mathcal{F}.
\end{equation}
\end{lemma}

\begin{proof}
Using $\chi=\chi_0+\ddbar u$ and \eqref{Theta 1}, we have
\[
\chi_{i\ov{i}} = (\chi_0)_{i\ov{i}}+u_{i\ov{i}}+2\mathrm{Re}\big(\Theta_{i\ov{i}}^ku_k\big).
\]
For \eqref{lem 1 eqn 1}, we compute
\[
F^{i\ov{i}}u_{i\ov{i}} = F^{i\ov{i}}\left(\chi_{i\ov{i}}-(\chi_0)_{i\ov{i}}-2\mathrm{Re}\big(\Theta_{i\ov{i}}^ku_k\big)\right).
\]
Recalling that $\underline{u}=0$, so $\underline{\chi}=\chi_0$. It then follows that
\[
F^{i\ov{i}}u_{i\ov{i}} = -F^{i\ov{i}}(\underline{\chi}_{i\ov{i}}-\chi_{i\ov{i}})
-2F^{i\ov{i}}\mathrm{Re}\big(\Theta_{i\ov{i}}^ku_k\big).
\]

For \eqref{lem 1 eqn 2}, we compute
\[
\begin{split}
F^{i\ov{i}}(|\de u|_g^2)_{i\ov{i}}
= {} & \sum_{p}F^{i\ov{i}}(|u_{ip}|^2+|u_{i\ov{p}}|^2)+\sum_{p}F^{i\ov{i}}\left(u_{p}u_{\ov{p}i\ov{i}}+u_{pi\ov{i}}u_{\ov{p}}\right) \\
\geq {} & \sum_{p}F^{i\ov{i}}(|u_{ip}|^2+|u_{i\ov{p}}|^2)+2\mathrm{Re}\left(F^{i\ov{i}}u_{p}u_{\ov{p}i\ov{i}}\right)-C\mathcal{F} \\
= {} & \sum_{p}F^{i\ov{i}}(|u_{ip}|^2+|u_{i\ov{p}}|^2)+2\mathrm{Re}\left(F^{i\ov{i}}u_{p}\left(u_{i\ov{i}\ov{p}}+O(1)\right)\right)-C\mathcal{F} \\
\geq {} & \sum_{p}F^{i\ov{i}}(|u_{ip}|^2+|u_{i\ov{p}}|^2)+2\mathrm{Re}\left(F^{i\ov{i}}u_{p}u_{i\ov{i}\ov{p}}\right)-C\mathcal{F}.
\end{split}
\]
Applying $\nabla_{\ov{p}}$ to \eqref{main eqn},
\[
F^{i\ov{i}}\chi_{i\ov{i}\ov{p}} = h_{\ov{p}},
\]
so
\[
F^{i\ov{i}}u_{i\ov{i}\ov{p}} = -F^{i\ov{i}}(\Theta_{i\ov{i}}^ku_{k\ov{p}}+\ov{\Theta_{i\ov{i}}^k}u_{\ov{k}\ov{p}})+h_{\ov{p}}+O(\mathcal{F}).
\]
It then follows that
\[
2\mathrm{Re}\left(F^{i\ov{i}}u_{p}u_{i\ov{i}\ov{p}}\right)
\geq -2\sum_{p}F^{i\ov{i}}\mathrm{Re}\left(\Theta_{i\ov{i}}^ku_{k\ov{p}}u_p+\ov{\Theta_{i\ov{i}}^k}u_{\ov{k}\ov{p}}u_p\right)-C\mathcal{F}.
\]
Thus,
\[
\begin{split}
F^{i\ov{i}}(|\de u|_g^2)_{i\ov{i}} \geq {} & \sum_{p}F^{i\ov{i}}(|u_{ip}|^2+|u_{i\ov{p}}|^2)-C\mathcal{F} \\
& -2\sum_{p}F^{i\ov{i}}\mathrm{Re}\left(\Theta_{i\ov{i}}^ku_{k\ov{p}}u_p+\ov{\Theta_{i\ov{i}}^k}u_{\ov{k}\ov{p}}u_p\right).
\end{split}
\]

For \eqref{lem 1 eqn 3}, we recall the following formulas for differentiating the largest eigenvalue of a matrix (see e.g. \cite[Lemma 5.2]{CTW19}):
\begin{equation}\label{formulas}
\begin{split}
\lambda_{1}^{\alpha\beta} {} & := \frac{\de\lambda_{1}}{\de\Phi_{\beta}^{\alpha}}=V_{1}^{\alpha}V_{1}^{\beta}, \\
\lambda_{1}^{\alpha\beta,\gamma\delta} {} & := \frac{\de^{2}\lambda_{1}}{\de\Phi_{\beta}^{\alpha}\de\Phi_{\delta}^{\gamma}}
= \sum_{\mu>1}\frac{V_{1}^{\alpha}V_{\mu}^{\beta}V_{\mu}^{\gamma}V_{1}^{\delta}+V_{\mu}^{\alpha}V_{1}^{\beta}V_{1}^{\gamma}V_{\mu}^{\delta}}
  {\lambda_{1}-\lambda_{\mu}}.
\end{split}
\end{equation}
Since $\Phi_{\beta}^{\alpha}=g^{\alpha\gamma}(u_{\gamma\beta}-B_{\gamma\beta})$ and $B_{\alpha\beta i}=0$ at $x_{0}$, then
\[
(\Phi_{\beta}^{\alpha})_{i} = u_{\alpha\beta i}-B_{\alpha\beta i} = u_{\alpha\beta i}.
\]
By \eqref{formulas}, we compute
\[
\begin{split}
F^{i\ov{i}}(\lambda_1)_{i\ov{i}}
= {} & F^{i\ov{i}}\lambda_{1}^{\alpha\beta,\gamma\delta}(\Phi_{\beta}^{\alpha})_i(\Phi_{\delta}^{\gamma})_{\ov{i}}+
F^{i\ov{i}}\lambda_{1}^{\alpha\beta}(\Phi_{\beta}^{\alpha})_{i\ov{i}} \\[1mm]
= {} & 2\sum_{\alpha>1}\frac{F^{i\ov{i}}|u_{V_{1}V_{\alpha}i}|^2}{\lambda_1-\lambda_\alpha}+F^{i\ov{i}}u_{V_{1}V_{1}i\ov{i}}+O(\mathcal{F}) \\
= {} & 2\sum_{\alpha>1}\frac{F^{i\ov{i}}|u_{V_{1}V_{\alpha}i}|^2}{\lambda_1-\lambda_\alpha}
+F^{i\ov{i}}\left(u_{i\ov{i}V_{1}V_{1}}+O(\lambda_1)\right)+O(\mathcal{F}) \\
\geq {} & 2\sum_{\alpha>1}\frac{F^{i\ov{i}}|u_{V_{1}V_{\alpha}i}|^2}{\lambda_1-\lambda_\alpha}+F^{i\ov{i}}u_{i\ov{i}V_{1}V_{1}}-C\lambda_1\mathcal{F}.
\end{split}
\]
Applying $\nabla_{V_1}\nabla_{V_1}$ to \eqref{main eqn},
\[
F^{i\ov{i}}\chi_{i\ov{i}V_{1}V_{1}}+F^{i\ov{j},p\ov{q}}\chi_{i\ov{j}V_{1}}\chi_{p\ov{q}V_{1}} = h_{V_{1}V_{1}},
\]
so
\[
\begin{split}
F^{i\ov{i}}u_{i\ov{i}V_{1}V_{1}}
={} & -F^{i\ov{j},p\ov{q}}\chi_{i\ov{j}V_{1}}\chi_{p\ov{q}V_{1}}
-2F^{i\ov{i}}\mathrm{Re}\big(\Theta_{i\ov{i}}^ku_{kV_1V_1}\big)+h_{V_1V_1}+O(\lambda_1\mathcal{F}) \\
\geq {} & -F^{i\ov{j},p\ov{q}}\chi_{i\ov{j}V_{1}}\chi_{p\ov{q}V_{1}}
-2F^{i\ov{i}}\mathrm{Re}\big(\Theta_{i\ov{i}}^ku_{V_1V_1k}\big)-C\lambda_1\mathcal{F}.
\end{split}
\]
Thus,
\[
\begin{split}
F^{i\ov{i}}(\lambda_1)_{i\ov{i}} \geq {} &
2\sum_{\alpha>1}\frac{F^{i\ov{i}}|u_{V_{1}V_{\alpha}i}|^2}{\lambda_1-\lambda_\alpha}
-F^{i\ov{j},p\ov{q}}\chi_{i\ov{j}V_{1}}\chi_{p\ov{q}V_{1}} \\
&-2F^{i\ov{i}}\mathrm{Re}\big(\Theta_{i\ov{i}}^ku_{V_1V_1k}\big)-C\lambda_1\mathcal{F}.
\end{split}
\]

For \eqref{lem 1 eqn 4}, recalling that $\rho=\nabla^{2}u+Lg$, we have
\[
\rho_{\alpha\beta i} = u_{\alpha\beta i}, \quad
\rho_{\alpha\beta i\ov{i}} = u_{\alpha\beta i\ov{i}}.
\]
We compute
\[
\begin{split}
F^{i\ov{i}}(|\rho|_{g}^{2})_{i\ov{i}}
= {} & 2\sum_{\alpha,\beta}F^{i\ov{i}}|u_{\alpha\beta i}|^{2}+2\sum_{\alpha,\beta}F^{i\ov{i}}\rho_{\alpha\beta}u_{\alpha\beta i\ov{i}} \\
= {} & 2\sum_{\alpha,\beta}F^{i\ov{i}}\left(|u_{i\alpha\beta}|^{2}+O(1)\right)
+2\sum_{\alpha,\beta}F^{i\ov{i}}\rho_{\alpha\beta}\left(u_{i\ov{i}\alpha\beta}+O(\lambda_1)\right) \\
\geq {} & 2\sum_{\alpha,\beta}F^{i\ov{i}}|u_{i\alpha\beta}|^{2}+2\sum_{\alpha,\beta}F^{i\ov{i}}\rho_{\alpha\beta}u_{i\ov{i}\alpha\beta}-CL^{2}\mathcal{F}.
\end{split}
\]
Applying $\nabla_{\beta}\nabla_{\alpha}$ to \eqref{main eqn},
\[
F^{i\ov{i}}\chi_{i\ov{i}\alpha\beta}+F^{i\ov{j},p\ov{q}}\chi_{i\ov{j}\alpha}\chi_{p\ov{q}\beta} = h_{\alpha\beta},
\]
then
\[
\begin{split}
F^{i\ov{i}}u_{i\ov{i}\alpha\beta}
= {} & -F^{i\ov{j},p\ov{q}}\chi_{i\ov{j}\alpha}\chi_{p\ov{q}\beta}-2F^{i\ov{i}}\mathrm{Re}(\Theta_{i\ov{i}}^ku_{k\alpha\beta})
+h_{\alpha\beta}+O(\lambda_1\mathcal{F}) \\
\geq {} & -F^{i\ov{j},p\ov{q}}\chi_{i\ov{j}\alpha}\chi_{p\ov{q}\beta}-2F^{i\ov{i}}\mathrm{Re}(\Theta_{i\ov{i}}^ku_{\alpha\beta k})
+h_{\alpha\beta}-CL\mathcal{F}
\end{split}
\]
which implies
\[
\begin{split}
& 2\sum_{\alpha,\beta}F^{i\ov{i}}\rho_{\alpha\beta}u_{i\ov{i}\alpha\beta} \\
\geq {} & -2\sum_{\alpha,\beta}F^{i\ov{j},p\ov{q}}\rho_{\alpha\beta}\chi_{i\ov{j}\alpha}\chi_{p\ov{q}\beta}
-4\sum_{\alpha,\beta}F^{i\ov{i}}\rho_{\alpha\beta}\mathrm{Re}(\Theta_{i\ov{i}}^ku_{\alpha\beta k})+2\sum_{\alpha,\beta}\rho_{\alpha\beta}h_{\alpha\beta}
-CL^{2}\mathcal{F} \\
\geq {} & -2\sum_{\alpha,\beta}F^{i\ov{j},p\ov{q}}\rho_{\alpha\beta}\chi_{i\ov{j}\alpha}\chi_{p\ov{q}\beta}
-4\sum_{\alpha,\beta}F^{i\ov{i}}\rho_{\alpha\beta}\mathrm{Re}(\Theta_{i\ov{i}}^ku_{\alpha\beta k})-CL^{2}\mathcal{F}.
\end{split}
\]
Here we used $\rho>0$, so constant $C$ does not depend on the upper bound of $\nabla^{2}h$ . Using the concavity of $F$ and $\rho>0$,
\[
-2\sum_{\alpha,\beta}F^{i\ov{j},p\ov{q}}\rho_{\alpha\beta}\chi_{i\ov{j}\alpha}\chi_{p\ov{q}\beta} \geq 0.
\]
This shows
\[
2\sum_{\alpha,\beta}F^{i\ov{i}}\rho_{\alpha\beta}u_{i\ov{i}\alpha\beta}
\geq -4\sum_{\alpha,\beta}F^{i\ov{i}}\rho_{\alpha\beta}\mathrm{Re}(\Theta_{i\ov{i}}^ku_{\alpha\beta k})-CL^{2}\mathcal{F},
\]
and so
\[
F^{i\ov{i}}(|\rho|_{g}^{2})_{i\ov{i}}
\geq 2\sum_{\alpha,\beta}F^{i\ov{i}}|u_{i\alpha\beta}|^{2}
-4\sum_{\alpha,\beta}F^{i\ov{i}}\rho_{\alpha\beta}\mathrm{Re}(\Theta_{i\ov{i}}^ku_{\alpha\beta k})-CL^2\mathcal{F}.
\]
\end{proof}

\begin{lemma}\label{lem 2}
At $x_0$, we have
\begin{equation}\label{lem 2 eqn 1}
\begin{split}
0 \geq {} & 2\sum_{\alpha>1}\frac{F^{i\ov{i}}|u_{V_{1}V_{\alpha}i}|^2}{\lambda_{1}(\lambda_{1}-\lambda_{\alpha})}
+\sum_{i\neq q}\frac{(F^{q\ov{q}}-F^{i\ov{i}})|\chi_{i\ov{q}V_1}|^2}{\lambda_1(\chi_{i\ov{i}}-\chi_{q\ov{q}})}
+\sum_{\alpha,\beta}\frac{F^{i\ov{i}}|u_{i\alpha\beta}|^{2}}{C_{A}\lambda_{1}^{2}} \\
& -\frac{F^{i\ov{i}}|u_{V_{1}V_{1}i}|^2}{\lambda_1^2}+\frac{1}{C}\sum_{p}F^{i\ov{i}}\left(|u_{ip}|^2+|u_{i\ov{p}}|^2\right) \\
& +\xi''\sum_{\alpha,\beta}F^{i\ov{i}}|(|\rho|_g^2)_i|^2+\eta''F^{i\ov{i}}|(|\de u|_g^2)_i|^2 \\
& +A^2e^{-Au}F^{i\ov{i}}|u_i|^2+Ae^{-Au}\sum_{i}F^{i\ov{i}}(\underline{\chi}_{i\ov{i}}-\chi_{i\ov{i}})-C\mathcal{F},
\end{split}
\end{equation}
where $C_A$ denotes a uniform constant depending on $A$.
\end{lemma}

\begin{proof}
Combining  \eqref{Q ii} and Lemma \ref{lem 1},
\begin{equation}\label{lem 2 eqn 2}
\begin{split}
0 \geq {} & 2\sum_{\alpha>1}\frac{F^{i\ov{i}}|u_{V_{1}V_{\alpha}i}|^2}{\lambda_{1}(\lambda_{1}-\lambda_{\alpha})}
-\frac{F^{i\ov{j},p\ov{q}}\chi_{i\ov{j}V_1}\chi_{p\ov{q}V_1}}{\lambda_1}-\frac{F^{i\ov{i}}|(\lambda_1)_i|^2}{\lambda_1^2} \\[1mm]
& +2\xi'\sum_{\alpha,\beta}F^{i\ov{i}}|u_{i\alpha\beta}|^{2}+\xi''\sum_{\alpha,\beta}F^{i\ov{i}}|(|\rho|_g^2)_i|^2 \\
& +\eta'\sum_{p}F^{i\ov{i}}\left(|u_{ip}|^2+|u_{i\ov{p}}|^2\right)+\eta''F^{i\ov{i}}|(|\de u|_g^2)_i|^2 \\
& +A^2e^{-Au}F^{i\ov{i}}|u_i|^2+Ae^{-Au}F^{i\ov{i}}(\underline{\chi}_{i\ov{i}}-\chi_{i\ov{i}})-(C+C\eta'+CL^{2}\xi')\mathcal{F} \\
& -\frac{2F^{i\ov{i}}\mathrm{Re}\big(\Theta_{i\ov{i}}^ku_{V_1V_1k}\big)}{\lambda_1}
-4\xi'\sum_{\alpha,\beta}F^{i\ov{i}}\rho_{\alpha\beta}\mathrm{Re}\big(\Theta_{i\ov{i}}^ku_{\alpha\beta k}\big) \\[2mm]
& -2\eta'\sum_{p}F^{i\ov{i}}\mathrm{Re}\left(\Theta_{i\ov{i}}^ku_{k\ov{p}}u_p+\ov{\Theta_{i\ov{i}}^k}u_{\ov{k}\ov{p}}u_p\right)
+2Ae^{-Au}F^{i\ov{i}}\mathrm{Re}\big(\Theta_{i\ov{i}}^ku_k\big).
\end{split}
\end{equation}
Thanks to \eqref{concave} and the concavity of $f$,
\[
\begin{split}
-\frac{F^{i\ov{j},p\ov{q}}\chi_{i\ov{j}V_1}\chi_{p\ov{q}V_1}}{\lambda_1}
= {} & -\sum_{i,q}\frac{f_{iq}\chi_{i\ov{i}V_1}\chi_{q\ov{q}V_1}}{\lambda_1}
-\sum_{i\neq q}\frac{(F^{i\ov{i}}-F^{q\ov{q}})|\chi_{i\ov{q}V_1}|^2}{\lambda_1(\chi_{i\ov{i}}-\chi_{q\ov{q}})} \\
\geq {} & \sum_{i\neq q}\frac{(F^{q\ov{q}}-F^{i\ov{i}})|\chi_{i\ov{q}V_1}|^2}{\lambda_1(\chi_{i\ov{i}}-\chi_{q\ov{q}})}.
\end{split}
\]
Using \eqref{formulas} and $B_{\alpha\beta i}=0$ at $x_{0}$, we see that
\begin{equation}\label{lambda 1 i}
(\lambda_{1})_{i} = \lambda_{1}^{\alpha\beta}(\Phi_{\beta}^{\alpha})_{i}
= V_{1}^{\alpha}V_{1}^{\beta}(u_{\alpha\beta i}-B_{\alpha\beta i}) = u_{V_{1}V_{1}i},
\end{equation}
and so
\[
-\frac{F^{i\ov{i}}|(\lambda_1)_i|^2}{\lambda_1^2}
= -\frac{F^{i\ov{i}}|u_{V_{1}V_{1}i}|^2}{\lambda_1^2}.
\]
By the definition of $\eta$, $\xi$ and $\hat{Q}$, we have
\[
\frac{1}{C} \leq \eta' \leq C, \quad
\frac{1}{18L^2} \leq \xi' \leq \frac{1}{3L^2}, \quad
\frac{L}{C_A} \leq \lambda_1 \leq L.
\]
From $\hat{Q}_{k}=0$, we obtain $-2F^{i\ov{i}}\mathrm{Re}(\Theta_{i\ov{i}}^{k}\hat{Q}_k)=0$. Combining this with $\ov{\Theta^{k}_{i\ov{i}}}=\Theta^{\ov{k}}_{i\ov{i}}$, we have that the last two lines of \eqref{lem 2 eqn 2} are zero. Substituting the above inequalities into \eqref{lem 2 eqn 2} gives \eqref{lem 2 eqn 1}.
\end{proof}

\subsection{Third order terms}
In this subsection, we deal with the third order terms in \eqref{lem 2 eqn 1}. The bad (non-positive) third order term is
\[
B := \frac{F^{i\ov{i}}|u_{V_{1}V_{1}i}|^2}{\lambda_1^2}.
\]
To control $B$, we define the set
\[
S = \{i\in\{1,2,\ldots,n-1\}~|~ F^{i\ov{i}} < A^{-2}e^{2Au(x_0)}F^{i+1\ov{i+1}} \}.
\]
Note that $S$ may be empty. Let $i_{0}$ be the maximal element of $S$. If $S=\emptyset$, then set $i_{0}=0$. Using $i_0$, we define another set
\[
I = \{i~|~ i > i_0 \}.
\]
Clearly, $I$ is not empty since $n\in I$. Roughly speaking, if $i\in I$, then $F^{i\ov{i}}$ is comparable with $F^{n\ov{n}}$. If $i\notin I$, then $F^{i\ov{i}}\ll F^{n\ov{n}}$. Specifically, note that if $i\in I$, then $i\not\in S$, so:
\[
F^{i\ov{i}}\geq A^{-2}e^{2Au}F^{i+1\ov{i+1}}.
\]
Since $i+1, i+2, \ldots, n \in I$ also, we get:
\begin{equation}\label{F ii}
\begin{split}
F^{i\ov{i}} & {} \geq A^{-4}e^{4Au}F^{i+2\ov{i+2}}\geq \ldots \\
& {} \geq A^{-2(n-i)}e^{2(n-i)Au}F^{n\ov{n}} \geq A^{-2n}e^{2nAu}F^{n\ov{n}}
\end{split}
\end{equation}
as $A\geq 1$ and $u\leq-1$. As we always have $F^{i\ov{i}}\leq F^{n\ov{n}}$, it follows that, for those $i\in I$, $F^{i\ov{i}}$ is comparable to $F^{n\ov{n}}$ by a uniform factor of $A^{-2n}e^{2nAu}$ (as $\|u\|_{L^\infty}$ is under control).

We now decompose the bad term $B$ into three terms based on $I$
\[
\begin{split}
B = {} & \sum_{i\notin I}\frac{F^{i\ov{i}}|u_{V_{1}V_{1}i}|^2}{\lambda_1^2}
+2\ve\sum_{i\in I}\frac{F^{i\ov{i}}|u_{V_{1}V_{1}i}|^2}{\lambda_1^2}
+(1-2\ve)\sum_{i\in I}\frac{F^{i\ov{i}}|u_{V_{1}V_{1}i}|^2}{\lambda_1^2} \\
=: {} & B_1+B_2+B_3,
\end{split}
\]
where $\ve\in\left(0,\frac{1}{2}\right)$ is a constant to be determined later.

\subsection*{$\bullet$ The terms $B_1$ and $B_2$}
Since $x_0$ is the maximum point of $\hat{Q}$, then $\hat{Q}_i(x_0)=0$. Combining this with $(\lambda_{1})_{i}=u_{V_{1}V_{1}i}$ (see \eqref{lambda 1 i}), we obtain
\begin{equation}\label{Q i}
\begin{split}
0 = \hat{Q}_{i}
= {} & \frac{(\lambda_{1})_i}{\lambda_1}+\xi'\cdot(|\rho|_g^2)_i+\eta'\cdot(|\de u|_g^2)_i-Ae^{-Au}u_i \\
= {} & \frac{u_{V_{1}V_{1}i}}{\lambda_1}+\xi'\cdot(|\rho|_g^2)_i+\eta'\cdot(|\de u|_g^2)_i-Ae^{-Au}u_i.
\end{split}
\end{equation}
We use \eqref{Q i} to deal with $B_1$ and $B_2$ as follows.

\begin{lemma}\label{bad 1 2}
At $x_0$, we have
\[
B_1+B_2 \leq \xi''F^{i\ov{i}}|(|\rho|_g^2)_i|^2+\eta''F^{i\ov{i}}|(|\de u|_g^2)_i|^2+6\ve A^{2}e^{-2Au}F^{i\ov{i}}|u_{i}|^{2}+C\mathcal{F}.
\]
\end{lemma}

\begin{proof}
By \eqref{Q i} and the Cauchy-Schwarz inequality, we see that
\[
\begin{split}
B_1 = {} & \sum_{i\notin I}F^{i\ov{i}}\left|\xi'\cdot(|\rho|_g^2)_i+\eta'\cdot(|\de u|_g^2)_i-Ae^{-Au}u_i\right|^{2} \\
\leq {} & 3(\xi')^{2}\sum_{i\notin I}F^{i\ov{i}}|(|\rho|_g^2)_i|^2+3(\eta')^{2}\sum_{i\notin I}F^{i\ov{i}}|(|\de u|_g^2)_i|^2
+3A^{2}e^{-2Au}\sum_{i\notin I}F^{i\ov{i}}|u_{i}|^{2} \\
\leq {} & 3(\xi')^{2}\sum_{i\notin I}F^{i\ov{i}}|(|\rho|_g^2)_i|^2+3(\eta')^{2}\sum_{i\notin I}F^{i\ov{i}}|(|\de u|_g^2)_i|^2
+CA^{2}e^{-2Au}\sum_{i\notin I}F^{i\ov{i}}.
\end{split}
\]
For $i\notin I$, by definition of $I$, we have $i\in S$ and so
\[
F^{i\ov{i}} < A^{-2}e^{2Au}F^{i+1\ov{i+1}} \leq A^{-2}e^{2Au}\mathcal{F},
\]
which implies
\[
B_1 \leq 3(\xi')^{2}\sum_{i\notin I}F^{i\ov{i}}|(|\rho|_g^2)_i|^2
+3(\eta')^{2}\sum_{i\notin I}F^{i\ov{i}}|(|\de u|_g^2)_i|^2+C\mathcal{F}.
\]
On the other hand,
\[
\begin{split}
B_2 \leq {} & 6\ve(\xi')^{2}\sum_{i\in I}F^{i\ov{i}}|(|\rho|_g^2)_i|^2+6\ve(\eta')^{2}\sum_{i\in I}F^{i\ov{i}}|(|\de u|_g^2)_i|^2
+6\ve A^{2}e^{-2Au}\sum_{i\in I}F^{i\ov{i}}|u_i|^2 \\
\leq {} & 6\ve(\xi')^{2}\sum_{i\in I}F^{i\ov{i}}|(|\rho|_g^2)_i|^2+6\ve(\eta')^{2}\sum_{i\in I}F^{i\ov{i}}|(|\de u|_g^2)_i|^2
+6\ve A^{2}e^{-2Au}F^{i\ov{i}}|u_i|^2.
\end{split}
\]
Combining the above inequalities and using $\ve\in\left(0,\frac{1}{2}\right)$, we obtain
\[
B_1+B_2 \leq 3(\xi')^{2}F^{i\ov{i}}|(|\rho|_g^2)_i|^2+3(\eta')^{2}F^{i\ov{i}}|(|\de u|_g^2)_i|^2
+6\ve A^{2}e^{-2Au}F^{i\ov{i}}|u_{i}|^{2}+C\mathcal{F}.
\]
Then Lemma \ref{bad 1 2} follows from $\xi''=3(\xi')^2$ and $\eta''=3(\eta')^2$.
\end{proof}

\subsection*{$\bullet$ The term $B_3$}
In \eqref{lem 2 eqn 1}, the good (non-negative) third order terms are
\[
G_1+G_2+G_3 := 2\sum_{\alpha>1}\frac{F^{i\ov{i}}|u_{V_{1}V_{\alpha}i}|^2}{\lambda_{1}(\lambda_{1}-\lambda_{\alpha})}
+\sum_{i\neq q}\frac{(F^{i\ov{i}}-F^{q\ov{q}})|\chi_{i\ov{q}V_1}|^2}{\lambda_1(\chi_{q\ov{q}}-\chi_{i\ov{i}})}
+\sum_{\alpha,\beta}\frac{F^{i\ov{i}}|u_{i\alpha\beta}|^{2}}{C_{A}\lambda_{1}^{2}}.
\]
We will use $G_1$, $G_{2}$ and $G_3$ to control $B_3$. Define
\[
W_{1} = \frac{1}{\sqrt{2}}(V_{1}-\sqrt{-1}JV_{1}), \quad
W_{1} = \sum_{q}\nu_{q}e_{q}, \quad
JV_{1} = \sum_{\alpha>1}\mu_{\alpha}V_{\alpha},
\]
where we used $V_1$ is orthogonal to $JV_1$. At $x_{0}$, since $V_{1}$ and $e_{q}$ are $g$-unit, then
\[
\sum_{q}|\nu_{q}|^{2} = 1, \quad
\sum_{\alpha>1}\mu_{\alpha}^{2} = 1.
\]

\begin{lemma}\label{bad 3 lem}
At $x_0$, we have
\[
|\nu_q| \leq \frac{C_A}{\lambda_1}, \quad \text{for $q\in I$},
\]
where $C_A$ is a uniform constant depending on $A$.
\end{lemma}

\begin{proof}
By \eqref{Q i} and the Cauchy-Schwarz inequality,
\[
B \leq 3(\xi')^{2}F^{i\ov{i}}|(|\rho|_g^2)_i|^2+3(\eta')^{2}F^{i\ov{i}}|(|\de u|_g^2)_i|^2+CA^{2}e^{-2Au}\mathcal{F}.
\]
Substituting this into \eqref{lem 2 eqn 1} and dropping non-negative terms $G_i$ ($i=1,2,3$), we see that
\[
\begin{split}
0 \geq {} & (\xi''-3(\xi')^{2})\sum_{\alpha,\beta}F^{i\ov{i}}|(|\rho|_g^2)_i|^2+(\eta''-3(\eta')^{2})F^{i\ov{i}}|(|\de u|_g^2)_i|^2 \\
& +\frac{1}{C}\sum_{p}F^{i\ov{i}}\left(|u_{ip}|^2+|u_{i\ov{p}}|^2\right)-CA^2e^{-2Au}\mathcal{F}.
\end{split}
\]
Thanks to $\xi''=3(\xi')^2$ and $\eta''=3(\eta')^2$,
\[
\sum_{p}F^{i\ov{i}}\left(|u_{ip}|^2+|u_{i\ov{p}}|^2\right) \leq C_A\mathcal{F}.
\]
For $i\in I$, using \eqref{F ii}, we have
\[
F^{i\ov{i}} \geq A^{-2n}e^{2nAu}F^{n\ov{n}},
\]
which implies
\[
\sum_{i\in I}\sum_{p}\left(|u_{ip}|^2+|u_{i\ov{p}}|^2\right) \leq C_A.
\]
Since $e_{i} = \frac{1}{\sqrt{2}}\left(\frac{\de}{\de x^{2i-1}}-\sqrt{-1}\frac{\de}{\de x^{2i}}\right)$ and $I=\{i_0+1,\ldots,n\}$, then
\begin{equation}\label{bad 3 lem eqn}
\sum_{\alpha=2i_0+1}^{2n}\sum_{\beta=1}^{2n}|u_{\alpha\beta}|^2 \leq C_{A}.
\end{equation}
Recalling that $V_{1}$ is the eigenvector of $\Phi$ corresponding to $\lambda_{1}$, we have
\[
V_{1}^{\alpha} = \frac{1}{\lambda_{1}}\sum_{\beta=1}^{2n}\Phi_{\beta}^{\alpha}V_{1}^{\beta}.
\]
Combining this with $\Phi_{\beta}^{\alpha}=g^{\alpha\gamma}(u_{\gamma\beta}-B_{\gamma\beta})$ and \eqref{bad 3 lem eqn},
\[
|V_{1}^{\alpha}| \leq \frac{1}{\lambda_{1}}\left(\sum_{\beta=1}^{2n}|u_{\alpha\beta}|+C\right)
\leq \frac{C_{A}}{\lambda_{1}}, \quad \text{for $\alpha=2i_0+1,\ldots,2n$}.
\]
It then follows that
\[
|\nu_{q}| \leq |V_{1}^{2q-1}|+|V_{1}^{2q}| \leq \frac{C_{A}}{\lambda_{1}}, \quad \text{for $q\in I$}.
\]
\end{proof}

We can now deal with the term $B_{3}$. By the definitions of $W_{1}$, $\nu_q$ and $\mu_\alpha$, we compute
\[
\begin{split}
u_{V_1V_1 i} = {} & -\sqrt{-1}u_{V_1JV_1 i}+\sqrt{2}u_{V_1\ov{W_1}i} \\[3mm]
= {} & -\sqrt{-1}\sum_{\alpha>1}\mu_\alpha u_{V_1 V_\alpha i}+\sqrt{2}\sum_{q}\ov{\nu_q}u_{V_{1}\ov{q}i} \\
= {} & -\sqrt{-1}\sum_{\alpha>1}\mu_\alpha u_{V_1 V_\alpha i}+\sqrt{2}\sum_{q}\ov{\nu_q}u_{i\ov{q}V_1}+O(1) \\
= {} & -\sqrt{-1}\sum_{\alpha>1}\mu_\alpha u_{V_1 V_\alpha i}+\sqrt{2}\sum_{q\notin I}\ov{\nu_q}u_{i\ov{q}V_1}
+\sqrt{2}\sum_{q\in I}\ov{\nu_q}u_{i\ov{q}V_1}+O(1).
\end{split}
\]
By \eqref{Theta},
\[
\chi_{i\ov{q}} = (\chi_0)_{i\ov{q}}+u_{i\ov{q}}+\Theta_{i\ov{q}}^ku_k+\Theta_{i\ov{q}}^{\ov{k}}u_{\ov{k}},
\]
we have
\[
u_{i\ov{q}V_1} = \chi_{i\ov{q}V_1}+O(\lambda_1).
\]
Thus,
\[
u_{V_1V_1 i} =  -\sqrt{-1}\sum_{\alpha>1}\mu_\alpha u_{V_1 V_\alpha i}+\sqrt{2}\sum_{q\notin I}\ov{\nu_q}\chi_{i\ov{q}V_1}
+\sqrt{2}\sum_{q\in I}\ov{\nu_q}u_{i\ov{q}V_1}+O(\lambda_1).
\]
Using the Cauchy-Schwarz inequality and Lemma \ref{bad 3 lem}, we decompose $B_{3}$ into several parts:
\[
\begin{split}
B_{3} = {} & (1-2\ve)\sum_{i\in I}\frac{F^{i\ov{i}}|u_{V_{1}V_{1}i}|^2}{\lambda_1^2} \\
\leq {} & (1-\ve)\sum_{i\in I}\frac{F^{i\ov{i}}}{\lambda_1^2}\left|\sum_{\alpha>1}\mu_\alpha u_{V_1 V_\alpha i}\right|^{2}
+\frac{C}{\ve}\sum_{i\in I}\sum_{q\notin I}\frac{F^{i\ov{i}}|\chi_{i\ov{q}V_1}|^{2}}{\lambda_1^2} \\
& +\frac{C_A}{\ve\lambda_1^2}\sum_{i\in I}\sum_{q\in I}\frac{F^{i\ov{i}}|u_{i\ov{q}V_1}|^{2}}{\lambda_1^2}+\frac{C\mathcal{F}}{\ve} \\
=: {} & B_{31}+B_{32}+B_{33}+\frac{C\mathcal{F}}{\ve}.
\end{split}
\]
\begin{lemma}\label{bad 3}
At $x_0$, we have
\[
B_3 \leq G_1+G_2+G_3+\frac{C\mathcal{F}}{\ve} ,
\]
assuming without loss of generality that $\lambda_1\geq\frac{C_A}{\ve}$ for some uniform constant $C_A$ depending on $A$.
\end{lemma}

\begin{proof}
We use $G_{i}$ to control $B_{3i}$ for $i=1,2,3$. For $B_{31}$, we compute
\[
\begin{split}
B_{31} \leq {} & (1-\ve)\sum_{i\in I}\frac{F^{i\ov{i}}}{\lambda_1^2}\left(\sum_{\alpha>1}(\lambda_{1}-\lambda_{\alpha})\mu_{\alpha}^{2}\right)
\left(\sum_{\alpha>1}\frac{|u_{V_{1}V_{\alpha}i}|^{2}}{\lambda_1-\lambda_\alpha}\right) \\
= {} & (1-\ve)\sum_{i\in I}\frac{F^{i\ov{i}}}{\lambda_1^2}\left(\lambda_{1}-\sum_{\alpha>1}\lambda_{\alpha}\mu_{\alpha}^{2}\right)
\left(\sum_{\alpha>1}\frac{|u_{V_{1}V_{\alpha}i}|^{2}}{\lambda_1-\lambda_\alpha}\right),
\end{split}
\]
where we used $\sum_{\alpha}\mu_\alpha^2=1$. On the other hand, since $\sup_{X}|\de u|_{g}$ and $\sup_{X}|\de\dbar u|_{g}$ are under control,
\[
\begin{split}
-C \leq {} & \chi(W_1, \ov{W_{1}}) = \chi_{0}(W_1, \ov{W_{1}})+(\ddbar u)(W_1, \ov{W_{1}}) \\
\leq {} & \nabla^{2}u(W_1, \ov{W_{1}})+C = \frac{1}{2}\left(u_{V_1V_1}+u_{JV_1 JV_1}\right)+C
= \frac{1}{2}\left(\lambda_1+\sum_{\alpha>1}\mu_\alpha^2\lambda_\alpha\right)+C,
\end{split}
\]
which implies
\[
\lambda_{1}-\sum_{\alpha>1}\lambda_{\alpha}\mu_{\alpha}^{2}
\leq 2\lambda_1+C \leq 2(1+\ve)\lambda_1,
\]
provided that $\lambda_1\geq\frac{C}{\ve}$. It then follows that
\[
B_{31} \leq 2(1-\ve^2)\sum_{i\in I}\sum_{\alpha>1}\frac{F^{i\ov{i}}|u_{V_{1}V_{\alpha}i}|^2}{\lambda_{1}(\lambda_{1}-\lambda_{\alpha})} \leq G_1.
\]

For $B_{32}$, we compute
\[
G_2 \geq \sum_{i\in I}\sum_{q\notin I}\frac{(F^{i\ov{i}}-F^{q\ov{q}})|\chi_{i\ov{q}V_1}|^2}{\lambda_1(\chi_{q\ov{q}}-\chi_{i\ov{i}})}
\geq \sum_{i\in I}\sum_{q\notin I}\frac{(F^{i\ov{i}}-F^{q\ov{q}})|\chi_{i\ov{q}V_1}|^2}{C\lambda_1}.
\]
For $i\in I$ and $q\notin I$, by definition of $I$,
\[
F^{q\ov{q}} \leq F^{i_0\ov{i_0}} \leq A^{-2}e^{2Au}F^{i_0+1\ov{i_0+1}} \leq A^{-2}e^{2Au}F^{i\ov{i}}.
\]
Since $A\gg1$ and $u\leq-1$:
\[
F^{i\ov{i}}-F^{q\ov{q}} \geq \frac{F^{i\ov{i}}}{2}.
\]
Thus,
\[
B_{32} = \frac{C}{\ve}\sum_{i\in I}\sum_{q\notin I}\frac{F^{i\ov{i}}|\chi_{i\ov{q}V_1}|^{2}}{\lambda_1^2}
\leq \frac{C}{\ve\lambda_1}\sum_{i\in I}\sum_{q\notin I}\frac{(F^{i\ov{i}}-F^{q\ov{q}})|\chi_{i\ov{q}V_1}|^2}{\lambda_1}
\leq G_{2}.
\]
as long as $\lambda_1\geq\frac{C}{\ve}$.

For $B_{33}$, we compute
\[
B_{33} \leq \frac{C_A}{\ve\lambda_1^2}\sum_{\alpha,\beta}\frac{F^{i\ov{i}}|u_{i\alpha\beta}|^{2}}{\lambda_1^2}
\leq \sum_{\alpha,\beta}\frac{F^{i\ov{i}}|u_{i\alpha\beta}|^{2}}{C_{A}\lambda_1^2} = G_3,
\]
provided that $\lambda_1\geq\frac{C_A}{\ve}$.
\end{proof}

\subsection{End of the Proof}
We now finish the proof of Theorem \ref{real Hessian estimate}.

\begin{proof}
Combining Lemma \ref{lem 2}, \ref{bad 1 2} and \ref{bad 3}, we obtain
\begin{equation}\label{proof eqn 1}
\begin{split}
0 \geq {} & \frac{1}{C_0}\sum_{p}F^{i\ov{i}}\left(|u_{ip}|^2+|u_{i\ov{p}}|^2\right)
+Ae^{-Au}\sum_{i}F^{i\ov{i}}(\underline{\chi}_{i\ov{i}}-\chi_{i\ov{i}}) \\
& +(A^2e^{-Au}-6\ve A^{2}e^{-2Au})F^{i\ov{i}}|u_i|^2-\frac{C_0\mathcal{F}}{\ve}.
\end{split}
\end{equation}
for a uniform constant $C_0$. Choose
\[
A = \frac{6C_0+1}{\kappa}, \quad \ve = \frac{e^{Au(x_0)}}{6}.
\]
According to Proposition \ref{subsolution prop}, there are then two possible cases:

\bigskip
\noindent
{\bf Case 1:} $\sum_{i}F^{i\ov{i}}(\underline{\chi}_{i\ov{i}}-\chi_{i\ov{i}})\geq\kappa\mathcal{F}$.

\bigskip

\noindent Combining this with our choice of $A$ and $\e$ implies that:
\[
\sum_{p}F^{i\ov{i}}\left(|u_{ip}|^2+|u_{i\ov{p}}|^2\right)+\mathcal{F} \leq 0,
\]
which is a contradiction.

\bigskip
\noindent
{\bf Case 2:} $F^{i\ov{i}}\geq\kappa\mathcal{F}$ for all $i$.

\bigskip

\noindent In this case, we obtain
\[
\sum_{p}F^{i\ov{i}}\left(|u_{ip}|^2+|u_{i\ov{p}}|^2\right) \leq C\mathcal{F}.
\]
It then follows that $\sum_{i,p}\left(|u_{ip}|^2+|u_{i\ov{p}}|^2\right)\leq C$, which implies $\lambda_1(x_0)\leq C$.
\end{proof}

\begin{remark}\label{boundary}
If $\de X\neq\emptyset$, then the above argument shows that
\[
\sup_{X}|\nabla^2 u|_g \leq \sup_{\de X}|\nabla^2 u|_g +C,
\]
where $C$ depends only on $(X,\omega)$, $\sup_{X}|u|$, $\sup_{X}|\de u|_g$, $\sup_{X}|\de\dbar u|_g$, $\chi_0$, $\underline{u}$, $\sigma_0$, $\sup_{X}h$, $\sup_{X}|\de h|_g$ and on a lower bound of $\nabla^2 h$.
\end{remark}

\section{Degenerate Equations}
In this short section, we give the proofs of Theorem \ref{degenerate estimate} and Corollary \ref{hessian}, which follow immediately from Theorem \ref{main estimate}.

\begin{theorem}
(Theorem \ref{degenerate estimate})  Suppose that $h\in C^2(X)$ satisfies $\sup_{\p\Gamma} f \leq h < \sup_\Gamma f$, and that $h_\e\in C^2(X)$ satisfy $\sup_{\p\Gamma} f < h_{\e} < \sup_\Gamma f$ and $h_\e \rightarrow h$ in $C^2$.  Let $\underline{u}$ be a $\mathcal{C}$-subsolution to the degenerate equation
\[
F(A(u)) = h(x),
\]
and let $u_\e$ be smooth solutions to the non-degenerate equations:
\[
F(A(u_\e)) = h_\e(x)
\]
for all $0 < \e$ sufficiently small. Then there exists a constant $C$, independent of $\e$ and depending only on $(X,\omega)$, $\chi_0$, $\underline{u}$, $\sigma_0$, $\sup_{X}h$, $\sup_{X}|\de h|_g$, and a lower bound of $\nabla^2 h$ such that
\[
\sup_{X}|u_\e|+\sup_{X}|\de u_\e|_g+\sup_{X}|\nabla^2 u_\e|_g \leq C,
\]
where $\nabla$ is the Levi-Civita connection of $g$.

In particular, up to extracting a subsequence, the $u_\e$ converge to a $u\in C^{1,1}(X)$ solving the degenerate version of equation \eqref{main eqn}:
\[
f^*(\lambda(u)) = h(x),
\]
where here $f^*$ is the upper semi-continuous extension of $f$ to $\ov{\Gamma}$.
\end{theorem}
\begin{proof}
By Proposition \ref{sigma 0}, we have that $\underline{u}$ is a $\mathcal{C}$-subsolution to $F(A(u_\e)) = h_\e(x)$ for all $\e>0$ sufficiently small. Thus, by Theorem \ref{main estimate} and the $C^{2}$ convergence of the $h_\e$, we have the required {\em a priori} estimate on the $u_\e$. The existence of $u$ follows then immediately from the Arzel\`a-Ascoli theorem.
\end{proof}

Before we prove Corollary \ref{hessian}, we briefly recall the necessary definitions. Fix an integer $m$ with $1\leq m\leq n$. Let $U$ be a domain in $\mathbb{C}^{n}$ and $\omega$ be a Hermitian metric on $\Omega$. Let $\Gamma_{m}(U)$ be the set of all $m$-positive smooth $(1,1)$-forms on $\Omega$ with respect to $\omega$ (we will always measure positivity using the fixed reference form $\omega$).

\begin{definition}
(\cite[Definition 2.10]{Lu13})  Suppose that $\theta\in\Gamma_m(U)$. We say an upper semi-continuous function $v:U\rightarrow[-\infty,+\infty)$ is $(\theta, m)$-subharmonic on $U$ if $v\in L_{\mathrm{loc}}^{1}(U)$ and
\begin{enumerate}
\item  $v + \rho$ is $\omega$-subharmonic for $\rho$ solving $i\p\pbar\rho\wedge\omega^{n-1} = \theta\wedge\omega^{n-1}$ and
\item for any $\gamma_{1},\ldots,\gamma_{m-1}\in\Gamma_{m}(U)$:
\[
(\theta+\ddbar v)\wedge\gamma_{1}\wedge\cdots\wedge\gamma_{m-1}\wedge\omega^{n-m} \geq 0
\]
in the sense of distributions.
\end{enumerate}
We write $\mSH(U,\theta)$ be the set of all $(\theta, m)$-subharmonic functions on $U$.
\end{definition}

For a compact Herimtian manifold $(X,\omega)$, the definition is similar. Let $A^{1,1}(X)$ be the space of smooth real $(1,1)$-forms on $X$. For any $\alpha\in A^{1,1}(X)$, write
\[
\sigma_{m}(\alpha)=\left(
\begin{matrix}
n\\
m
\end{matrix}
\right)
\frac{\alpha^{m}\wedge\omega^{n-m}}{\omega^{n}}.
\]
Then we define
\[
\Gamma_{m}(X,\omega)=\{ \theta\in A^{1,1}(X)~|~ \text{$\sigma_{l}(\theta)>0$ for $l=1,2,\ldots,m$} \},
\]
to be the cone of (strictly) $m$-positive forms on $X$. As $\omega$ is fixed, we will often write $\Gamma_{m}(X)$ instead of $\Gamma_{m}(X,\omega)$.

\begin{definition}\label{Def mSH}
Let $(X,\omega)$ be a compact Hermitian manifold and $\theta\in\Gamma_{m}(X)$. We say an upper semi-continuous function $v:X\rightarrow[-\infty,+\infty)$ is $(\theta, m)$-subharmonic on $X$ if  $v\in L^{1}(X)$ and $v|_{U}\in\mSH(U,\theta)$ for any local coordinate system $U$. Write $\mSH(X,\theta)$ for the set of all $(\theta,m)$-subharmonic functions on $X$.
\end{definition}

\begin{corollary}
(Corollary \ref{hessian}) Let $(X,\omega)$ be a compact Hermitian manifold and $\theta\in \Gamma_m(X)$ for some $1 \leq m \leq n$. Then for any non-negative function $h$ on $X$ such that $\int_{X}h\omega^n>0$ and $h^{\frac{1}{m}}\in C^{2}(X)$, there exists a pair $(u,c)\in C^{1,1}(X)\times\mathbb{R}_{+}$ such that
\[
\begin{cases}
(\theta+\ddbar u)^{m}\wedge \omega^{n-m} = ch\omega^{n}, \\
u\in\mSH(X,\theta),  \quad \sup_{X}u = -1.
\end{cases}
\]
\end{corollary}

\begin{proof}
Write $h_i=h+i^{-1}$. Since it is well known that $\sigma_{m}^{\frac{1}{m}}$ falls under the frame work of \eqref{main eqn}, all we need to do to apply Theorem \ref{degenerate estimate} is show that there exist pairs $(u_i, c_i) \in C^\infty(X)\times\R_+$ solving:
\[
\begin{cases}
(\theta+\ddbar u_{i})^{m}\wedge \omega^{n-m} = c_{i}h_{i}\omega^{n}, \\
u_{i}\in\mSH(X,\theta),  \quad \sup_{X}u_{i} = -1.
\end{cases}
\]
and such that the $c_i$ are also uniformly bounded in $i$:
\[
\frac{1}{C} \leq c_i \leq C.
\]

The existence of such pair follows from \cite[Proposition 21]{Szekelyhidi18} or \cite[Theorem 1.1]{Zhang17}. The boundedness of the $c_i$ follows from \cite[Lemma 3.13]{KN16}. For the reader's convenience, we repeat the short proof here. The lower bound for the $c_i$ follows immediatly from applying the maximum principle at the minimum of $u_{i}$. Then, weak compactness of $\sup$-normalized $(\theta, m)$-subharmonic functions implies that the $L^1$-norm of the $u_i$ is also controlled:
\[
\|u_{i}\|_{L^{1}} \leq C.
\]
By Maclaurin's inequality and Stokes' formula, we have
\[
\begin{split}
\int_{X}(c_{i}h_{i})^{\frac{1}{m}}\omega^{n} \leq C\int_{X}\sigma_{1}(\theta+\ddbar u_{i})\omega^{n} \leq C+C\|u_{i}\|_{L^{1}} \leq C.
\end{split}
\]
Combining this with $\int_{X}h_{i}^{\frac{1}{m}}\omega^{n} \geq C^{-1}$, we obtain
\[
c_{i} \leq C,
\]
as desired. Thus, after possibly taking a subsequence, we get that $c_i \rightarrow c > 0$. Since the $(c_i h_i)^{1/m}$ are non-degenerate and converge to $(ch)^{1/m}$ in $C^{2}$, we can thus apply Theorem \ref{degenerate estimate} to conclude the existence of $u$.
\end{proof}


\section{$C^{1,1}$ regularity of $(m,\theta)$-envelopes}
In this section, we prove Theorem \ref{envelope}:
\begin{theorem}\label{envelope section}
(Theorem \ref{envelope}) Let $(X,\omega)$ be a compact Hermitian manifold and $\theta\in \Gamma_m(X)$. If $h\in C^{1,1}(X)$, we define the envelope:
\[
P_{m, \theta}(h) := \sup\{ v\in\mSH(X, \theta)~|~v\leq h \}.
\]
Then $P_{m, \theta}(h)\in C^{1,1}(X)$. In particular, $P_{m,\theta}(h)$ solves:
\begin{equation}\label{env equ}
(\theta + i\p\pbar P_{m,\theta}(h))^m\wedge\omega^{n-m} = \chi_{K} \theta_h^m\wedge\omega^{n-m},
\end{equation}
where $\theta_{h}=\theta+i\p\pbar h$ and $K = \{ P_{m,\theta}(h) = h\}$ is the contact set.
\end{theorem}

Theorem \ref{envelope section} will follow from some simple modifications to the estimates in Theorem \ref{main estimate}, as well as some standard results about $(\theta, m)$-subharmonic functions, which we recall for completeness.

\subsection{Some estimates}

Following the method of Berman \cite[Section 2.1]{Berman19}, we consider the following equation
\begin{equation}\label{envelope CHE 1}
\begin{cases}
\log\sigma_{m}(\theta+\ddbar u) = \frac{1}{\ve}(u-h), \\[1mm]
\chi = \theta+\ddbar u \in \Gamma_{m}(X),
\end{cases}
\end{equation}
where $\ve\in(0,1)$ and $h\in C^{\infty}(X)$. We derive some estimates for \eqref{envelope CHE 1} that are essentially the same as those in Theorem \ref{main estimate}.

\begin{lemma}\label{envelope estimates}
Let $u$ be a smooth solution of \eqref{envelope CHE 1}. Then there exists a constant $C$ depending only on $\|h\|_{C^{2}}$, $\theta$ and $(X,\omega)$ such that
\begin{enumerate}
\item $\sup_{X}(u-h)\leq C\ve$, $\inf_{X}u\geq-C$.
\item $\sup_X|\de u|_g+\sup_X|\de\dbar u|_g\leq C$.
\item $\sup_X|\nabla^2 u|_g\leq C$.
\end{enumerate}
\end{lemma}

\begin{proof}
For (1), let $x_0$ be the maximum point of $(u-h)$. At $x_0$, we have
\[
u-h \leq \ve\log\sigma_{m}(\theta+\ddbar h) \leq C\ve,
\]
which implies $\sup_X(u-h)\leq C\ve$. Let $y_0$ be the minimum point of $u$. At $y_0$, we have
\[
u-h \geq \ve\log\sigma_{m}(\theta) \geq -C.
\]
which implies $\inf_X u\geq-C$.

(2) and (3) can be proved by applying arguments similar to \cite[Theorem 2]{Szekelyhidi18} and Theorem \ref{main estimate}, respectively, with only minor modifications to deal with the dependency of the right hand side on the unknown function $u$. The modifications are standard adaptations from the Monge-Amp\`ere case \cite{Aubin76,Yau78}, and are well known -- for the reader's convenience, we give a sketch of them in the context of proving (3). In this setting,
\[
f = \log\sigma_m, \quad
\underline{u} = 0, \quad
\underline{\chi} = \chi_0 = \theta.
\]
Note that $\log\sigma_{m}$ satisfies all the requisite assumptions of Theorem \ref{main estimate}.

We define a similar maximum principle quantity:
\[
Q = \log\lambda_1(\nabla^{2}u)+\xi(|\rho|_g^2)+\eta(|\de u|_g^2)+e^{-A(u-B)}.
\]
where $B$ is a constant such that $\sup_{X}(u-B)\leq-1$. We also modify the definitions of $\xi$ and $\eta$ slightly:
\[
\xi(s) = -\frac{1}{3}\log\left(100n^{2}L^{2}-s\right), \quad L = \sup_{X}|\nabla^{2}u|_g+1,
\]
\[
\eta(s) = -\frac{1}{3}\log\left(1+4\sup_{X}|\de h|_g^2+\sup_{X}|\de u|_g^2-s\right).
\]
Let $x_0$ be the maximum point of $Q$, and $\hat{Q}$ be the corresponding perturbed quantity. Near $x_0$, we choose the same coordinates $\{x^{\alpha}\}_{\alpha=1}^{2n}$ and local frame $\{e_{i}\}_{i=1}^{n}$ as in the proof of Theorem \ref{real Hessian estimate}. Equation \eqref{envelope CHE 1} can be written as
\begin{equation}\label{envelope CHE 2}
F(A) = \frac{1}{\ve}(u-h), \quad A_j^i = g^{i\ov{p}}\chi_{j\ov{p}}.
\end{equation}
As mentioned, the only difference between \eqref{main eqn} and \eqref{envelope CHE 2} is the right hand side. To apply the maximum principle at $x_0$, we need to establish some inequalities. We begin by showing that the conclusion of Lemma \ref{lem 2} still holds for \eqref{envelope CHE 2}.

The following calculations are very similar to Lemma \ref{lem 1}. It is clear that
\begin{equation}\label{envelope eqn 1}
F^{i\ov{i}}u_{i\ov{i}} = -F^{i\ov{i}}(\underline{\chi}_{i\ov{i}}-\chi_{i\ov{i}})
-2F^{i\ov{i}}\mathrm{Re}\big(\Theta_{i\ov{i}}^ku_k\big).
\end{equation}
Applying $\nabla_{\ov{p}}$ to equation \eqref{envelope CHE 2},
\[
F^{i\ov{i}}\chi_{i\ov{i}\ov{p}} = \frac{1}{\ve}(u_{\ov{p}}-h_{\ov{p}}),
\]
which implies
\[
\begin{split}
2\mathrm{Re}\left(F^{i\ov{i}}u_{p}u_{i\ov{i}\ov{p}}\right)
\geq {} & -2\sum_{p}F^{i\ov{i}}\mathrm{Re}\left(\Theta_{i\ov{i}}^ku_{k\ov{p}}u_p+\ov{\Theta_{i\ov{i}}^k}u_{\ov{k}\ov{p}}u_p\right) \\
& +\frac{2}{\ve}(|\de u|_g^2-|\de u|_g|\de h|_g)-C\mathcal{F}.
\end{split}
\]
Combining this with the Cauchy-Schwarz inequality,
\begin{equation}\label{envelope eqn 2}
\begin{split}
F^{i\ov{i}}(|\de u|_g^2)_{i\ov{i}}
\geq {} & \sum_{p}F^{i\ov{i}}(|u_{ip}|^2+|u_{i\ov{p}}|^2)+2\mathrm{Re}\left(F^{i\ov{i}}u_{p}u_{i\ov{i}\ov{p}}\right)-C\mathcal{F} \\
\geq {} & \sum_{p}F^{i\ov{i}}(|u_{ip}|^2+|u_{i\ov{p}}|^2)-\frac{1}{\ve}|\de h|_g^2-C\mathcal{F} \\
& -2\sum_{p}F^{i\ov{i}}\mathrm{Re}\left(\Theta_{i\ov{i}}^ku_{k\ov{p}}u_p+\ov{\Theta_{i\ov{i}}^k}u_{\ov{k}\ov{p}}u_p\right).
\end{split}
\end{equation}
Applying $\nabla_{V_1}\nabla_{V_1}$ to \eqref{envelope CHE 2}, we get
\[
F^{i\ov{i}}\chi_{i\ov{i}V_{1}V_{1}}+F^{i\ov{j},p\ov{q}}\chi_{i\ov{j}V_{1}}\chi_{p\ov{q}V_{1}} = \frac{1}{\ve}(\lambda_1-h_{V_{1}V_{1}})
\geq \frac{1}{2\ve}\lambda_1,
\]
where we assume without loss of generality that $\lambda_1-h_{V_{1}V_{1}}\geq\frac{1}{2}\lambda_1$. Thus,
\[
F^{i\ov{i}}u_{i\ov{i}V_{1}V_{1}}
\geq -F^{i\ov{j},p\ov{q}}\chi_{i\ov{j}V_{1}}\chi_{p\ov{q}V_{1}}
-2F^{i\ov{i}}\mathrm{Re}\big(\Theta_{i\ov{i}}^ku_{V_1V_1k}\big)
+\frac{1}{2\ve}\lambda_1-C\lambda_1\mathcal{F},
\]
and so
\begin{equation}\label{envelope eqn 3}
\begin{split}
F^{i\ov{i}}(\lambda_1)_{i\ov{i}}
\geq {} & 2\sum_{\alpha>1}\frac{F^{i\ov{i}}|u_{V_{1}V_{\alpha}i}|^2}{\lambda_1-\lambda_\alpha}+F^{i\ov{i}}u_{i\ov{i}V_{1}V_{1}}-C\lambda_1\mathcal{F} \\
\geq {} & 2\sum_{\alpha>1}\frac{F^{i\ov{i}}|u_{V_{1}V_{\alpha}i}|^2}{\lambda_1-\lambda_\alpha}
-F^{i\ov{j},p\ov{q}}\chi_{i\ov{j}V_{1}}\chi_{p\ov{q}V_{1}} \\
&-2F^{i\ov{i}}\mathrm{Re}\big(\Theta_{i\ov{i}}^ku_{V_1V_1k}\big)+\frac{1}{2\ve}\lambda_1-C\lambda_1\mathcal{F}.
\end{split}
\end{equation}
Applying $\nabla_{\beta}\nabla_{\alpha}$ to \eqref{envelope CHE 2} gives,
\[
F^{i\ov{i}}\chi_{i\ov{i}\alpha\beta}+F^{i\ov{j},p\ov{q}}\chi_{i\ov{j}\alpha}\chi_{p\ov{q}\beta} = \frac{1}{\ve}(u_{\alpha\beta}-h_{\alpha\beta}),
\]
which implies
\[
\begin{split}
& 2\sum_{\alpha,\beta}F^{i\ov{i}}\rho_{\alpha\beta}u_{i\ov{i}\alpha\beta} \\
\geq {} & -2\sum_{\alpha,\beta}F^{i\ov{j},p\ov{q}}\rho_{\alpha\beta}\chi_{i\ov{j}\alpha}\chi_{p\ov{q}\beta}
-4\sum_{\alpha,\beta}F^{i\ov{i}}\rho_{\alpha\beta}\mathrm{Re}(\Theta_{i\ov{i}}^ku_{\alpha\beta k}) \\
& +\frac{2}{\ve}\sum_{\alpha,\beta}\rho_{\alpha\beta}(u_{\alpha\beta}-h_{\alpha\beta})-CL^{2}\mathcal{F} \\
\geq {} & -4\sum_{\alpha,\beta}F^{i\ov{i}}\rho_{\alpha\beta}\mathrm{Re}(\Theta_{i\ov{i}}^ku_{\alpha\beta k})
+\frac{2}{\ve}\sum_{\alpha,\beta}\rho_{\alpha\beta}(\rho_{\alpha\beta}-Lg_{\alpha\beta}-h_{\alpha\beta})-CL^{2}\mathcal{F} \\
\geq {} & -4\sum_{\alpha,\beta}F^{i\ov{i}}\rho_{\alpha\beta}\mathrm{Re}(\Theta_{i\ov{i}}^ku_{\alpha\beta k})-\frac{24n^2}{\ve}L^2-CL^{2}\mathcal{F},
\end{split}
\]
where we assume without loss of generality that $L\geq\sup_{X}|\nabla^2 h|_g$. It then follows that
\begin{equation}\label{envelope eqn 4}
\begin{split}
& F^{i\ov{i}}(|\rho|_{g}^{2})_{i\ov{i}} \\[2mm]
\geq {} & 2\sum_{\alpha,\beta}F^{i\ov{i}}|u_{i\alpha\beta}|^{2}
+2\sum_{\alpha,\beta}F^{i\ov{i}}\rho_{\alpha\beta}u_{i\ov{i}\alpha\beta}-CL^{2}\mathcal{F} \\
\geq {} & 2\sum_{\alpha,\beta}F^{i\ov{i}}|u_{i\alpha\beta}|^{2}-4F^{i\ov{i}}\rho_{\alpha\beta}\mathrm{Re}(\Theta_{i\ov{i}}^ku_{\alpha\beta k})
-\frac{24n^2}{\ve} L^2-CL^{2}\mathcal{F}.
\end{split}
\end{equation}

We can now combine \eqref{envelope eqn 1}, \eqref{envelope eqn 2}, \eqref{envelope eqn 3}, and \eqref{envelope eqn 4} and apply the same argument as in Lemma \ref{lem 2} to obtain:
\[
0 \geq \text{(right hand side of \eqref{lem 2 eqn 1})}+\frac{1}{\ve}\left(\frac{1}{2}-24n^2L^2\xi'-\eta'|\de h|_g^2\right).
\]
But by the definition of $\xi$ and $\eta$, we have
\[
\xi' \leq \frac{1}{100n^2L^2}, \quad \eta' \leq \frac{1}{4|\de h|_g^2+1},
\]
so
\[
\frac{1}{\ve}\left(\frac{1}{2}-n^2L^2\xi'-\eta'|\de h|_g^2\right) \geq 0.
\]
Hence, the conclusion of Lemma \ref{lem 2} still applies for equation \eqref{envelope CHE 2}.

We can now continue the argument as in Section 4 until we get to equation \eqref{proof eqn 1} -- in our new setting, this becomes the slightly modified inequality:
\[
\begin{split}
0 \geq {} & \frac{1}{C_0}\sum_{p}F^{i\ov{i}}\left(|u_{ip}|^2+|u_{i\ov{p}}|^2\right)
+Ae^{-A(u-B)}F^{i\ov{i}}(\underline{\chi}_{i\ov{i}}-\chi_{i\ov{i}}) \\
& +(A^2e^{-A(u-B)}-6\ve' A^{2}e^{-2A(u-B)})F^{i\ov{i}}|u_i|^2-\frac{C_0\mathcal{F}}{\ve'}-C_{0},
\end{split}
\]
where we have replaced the $\e$ in \eqref{proof eqn 1} with $\e'$, to avoid confusion with the $\e$ already being used in this section. We now need to apply Proposition \ref{subsolution prop} to finish. Since the right hand side of the equation \eqref{envelope CHE 2} depends on $\ve$, we cannot apply Proposition \ref{subsolution prop} directly, as we need the real Hessian estimate to be independent of $\ve$. Instead, note that there exists $\kappa>0$ depending only on $(X,\omega)$ and $\theta$ such that
\[
\theta-\kappa\omega \in \Gamma_{m}(X).
\]
By $\underline{\chi}=\theta$ and G{\aa}rding's inequality,
\[
F^{i\ov{i}}(\underline{\chi}_{i\ov{i}}-\chi_{i\ov{i}})
= F^{i\ov{i}}\theta_{i\ov{i}}-m
= F^{i\ov{i}}(\theta_{i\ov{i}}-\kappa g_{i\ov{i}}+\kappa g_{i\ov{i}})-m
\geq \kappa\mathcal{F}-m.
\]
Choosing then
\[
A = \frac{6C_0+1}{\kappa}, \quad \ve' = \frac{e^{A(u(x_0)-B)}}{6},
\]
it follows that
\[
\sum_{p}F^{i\ov{i}}\left(|u_{ip}|^2+|u_{i\ov{p}}|^2\right)+\mathcal{F} \leq C.
\]
Thanks to \cite[Lemma 2.2 (2)]{HMW10}, $\mathcal{F}\leq C$ implies $F^{i\ov{i}} \geq \frac{1}{C}$. We thus get:
\[
\sum_{p}\left(|u_{ip}|^2+|u_{i\ov{p}}|^2\right) \leq C,
\]
which shows that $\lambda_1\leq C$, as required.
\end{proof}

\subsection{$(\theta, m)$-subharmonic functions}

We now recall some basic facts about $(\theta, m)$-subharmonic functions. In what follows, we will write:
\[
P_\theta(h) := P_{m, \theta}(h),
\]
for simplicity.

\begin{lemma}\label{viscosity}
Let $v\in\mSH(X,\theta)$ and $p\in X$. Suppose that $U$ is a coordinate system centered at $p$. If $\vp\in C^{2}(U)$ satisfies
\[
\vp(p) = v(p), \quad \vp \geq v \ \text{on $U$},
\]
then for any $\gamma_{1},\ldots,\gamma_{m-1}\in\Gamma_{m}(U)$,
\[
(\theta+\ddbar \vp)\wedge\gamma_{1}\wedge\cdots\wedge\gamma_{m-1}\wedge\omega^{n-m} \geq 0 \ \ \text{at $p$}.
\]
In particular, $\sigma_{m}(\theta+\ddbar\vp) \geq 0$ at $p$.
\end{lemma}
\begin{proof}
We follow the argument of \cite[Lemma 7]{Plis13}. Since $\gamma_{1},\ldots,\gamma_{k-1}\in\Gamma_{m}(U)$, we have that
\[
\gamma_{1}\wedge\cdots\wedge\gamma_{m-1}\wedge\omega^{n-m}
\]
is a positive $(n-1,n-1)$-form. By \cite[(4.8)]{Michelsohn82}, there exists a Hermitian metric $\beta$ on $U$ such that
\[
\beta^{n-1} = \gamma_{1}\wedge\cdots\wedge\gamma_{m-1}\wedge\omega^{n-m}.
\]
After shrinking $U$, we can find a bounded $f\in C^{\infty}(U)$ solving the linear equation
\[
\theta\wedge\beta^{n-1} = \ddbar f\wedge\beta^{n-1}.
\]
Combining this with $v\in \mSH(X,\theta)$, we obtain
\[
\ddbar(f+v)\wedge\beta^{n-1} = (\theta+\ddbar v)\wedge\gamma_{1}\wedge\cdots\wedge\gamma_{m-1}\wedge\omega^{n-m} \geq 0 \ \ \text{on $U$}
\]
in the sense of distributions. Clearly, $f+\vp\in C^{2}(U)$ satisfies
\[
(f+\vp)(p) = (f+v)(p), \quad f+\vp \geq f+v \ \text{on $U$}.
\]
By \cite[Theorem 9.2 and 9.3]{HL13}, we have
\[
\ddbar(f+\vp)\wedge\beta^{n-1} \geq 0 \ \ \text{at $p$}.
\]
It then follows that
\[
(\theta+\ddbar \vp)\wedge\gamma_{1}\wedge\cdots\wedge\gamma_{m-1}\wedge\omega^{n-m} \geq 0 \ \ \text{at $p$},
\]
as desired.

The last claim, that $\sigma_m(\theta + \ddbar\vp)\geq 0$ at $p$, nows follows from a simple application of G\aa rding's inequality.
\end{proof}

Let $\e \in (0,1)$. By Lemma \ref{envelope estimates} and the continuity method in \cite{Szekelyhidi18}, there exists a smooth function $u_{\ve}\in\mSH(X, \theta)$ solving \eqref{envelope CHE 1}:
\begin{equation}
\begin{cases}
\log\sigma_{m}(\theta+\ddbar u_{\ve}) = \frac{1}{\ve}(u_{\ve}-h), \\[1mm]
\theta+\ddbar u \in \Gamma_{m}(X).
\end{cases}
\end{equation}

\begin{lemma}\label{envelope lem}
Let $0 < \delta < 1$ and $\theta\in\Gamma_m(X)$. Then for any $v\in\mSH(X,(1-\delta)\theta)$ such that $v\leq h$, we have
\[
v \leq u_{\ve}+\delta \ \  \text{on $X$},
\]
for all $\e > 0$ sufficiently small. In particular, $P_{(1-\delta)\theta}(h)\leq u_{\ve}+\delta$.
\end{lemma}

\begin{proof}
The argument is extracted from \cite[Theorem 3.2]{LN15}. Let $p$ be the minimum point of the lower semi-continuous function $u_{\ve}-v$, and set
\[
B = (u_{\ve} - v)(p) = \min_{X}(u_{\ve} - v).
\]
It suffices to prove $B\geq -\delta$. Pick a coordinate system $U\subset X$ centered at $p$, and set
\[
\vp = u_{\ve} - B \in C^\infty(X).
\]
It is clear that
\[
\vp(p) = v(p), \quad \vp \geq v \ \ \text{on $U$}.
\]
By Lemma \ref{viscosity}, we have
\[
\sigma_{m}((1-\delta)\theta+\ddbar\vp) \geq 0 \ \ \text{at $p$},
\]
which implies
\[
\sigma_{m}(\theta+\ddbar u_\e) \geq \delta^{m}\sigma_{m}(\theta) > 0 \ \ \text{at $p$}.
\]
Using $v\leq h$ and \eqref{envelope CHE 2}, we get:
\[
v(p) - \delta \leq  h(p) + \e\log(\delta^m\sigma_m(\theta)) \leq u_\e(p)
\]
as long as $\e$ is sufficiently small, depending only on $\theta$ and $\delta$. Rearranging yields $B\geq -\delta$, as desired.
\end{proof}

\subsection{Proof of Theorem \ref{envelope section}}
We now prove Theorem \ref{envelope section}:

\begin{proof}[Proof of Theorem \ref{envelope section}]
We begin by showing that it suffices to prove the theorem for $h\in C^\infty(X)$. To see this, suppose that $h\in C^{1,1}(X)$ and pick a sequence of smooth functions $h_i$ such that
\[
\lim_{i\rightarrow\infty}\|h_{i}-h\|_{C^0} = 0, \quad
\|h_{i}\|_{C^{2}} \leq C,
\]
where $C$ is independent of $i$. By the definition of the envelope, we have
\[
\|P_\theta(h_{i})-P_\theta(h)\|_{L^{\infty}} \leq \|h_{i}-h\|_{C^{0}},
\]
which implies
\[
\lim_{i\rightarrow\infty}\|P_\theta(h_{i})-P_\theta(h)\|_{L^{\infty}} = 0.
\]
Thus, the theorem will follow if we can prove that $\|P_\theta(h_i)\|_{C^2} \leq C$ for the smooth functions $h_i$.

Suppose then that $h\in C^\infty(X)$. As already noted, by Lemma \ref{envelope estimates} and the continuity method in \cite{Szekelyhidi18}, there exists smooth functions $u_{\ve}\in\mSH(X,\theta)$ solving
\[
\log\sigma_{m}(\theta+\ddbar u_{\ve}) = \frac{1}{\ve}(u_{\ve}-h),
\]
for any $\e \in (0,1)$. We claim
\begin{equation}\label{envelope claim}
\lim_{\ve\rightarrow0}\|u_{\ve}-P_\theta(h)\|_{C^0} = 0.
\end{equation}
Given this, the theorem follows immediately from Lemma \ref{envelope estimates} and the Arzel\`a-Ascoli theorem.

To prove \eqref{envelope claim}, by (1) of Lemma \ref{envelope estimates}, we have
\[
u_{\ve} - C\ve \leq h.
\]
Since $u_{\ve}-C\ve\in\mSH(X, \theta)$ also, we thus have:
\begin{equation}\label{envelope eqn 5}
u_{\ve}-C\ve \leq P_\theta(h).
\end{equation}
Combining this with Lemma \ref{envelope lem} now gives the two-sided bound:
\[
P_{(1-\delta)\theta}(h) - \delta \leq u_\e \leq P_\theta(h) + C\e,
\]
for all $0 < \e$ sufficiently small. By (2) of Lemma \ref{envelope estimates}, any subsequence $\e_i \rightarrow 0$ has a further subsequence (which we will also denote by $\e_i$), such that $u_{\e_i} \rightarrow u_0\in C^0(X)$ uniformly as $i\rightarrow\infty$. Taking $i\rightarrow\infty$ in the above bound gives:
\[
P_{(1-\delta)\theta}(h) - \delta \leq u_0 \leq P_\theta(h),
\]
But it is clear from the definition that $\lim_{\delta\rightarrow 0} P_{(1-\delta)\theta}(h) = P_\theta(h)$, so we can now let $\delta\rightarrow 0$ to get:
\[
u_0 = P_\theta(h).
\]
Thus, any $C^0$-convergent subsequence of $\{u_\e\}$ converges to $P_\theta(h)$, so we conclude that $u_\e\xrightarrow{C^0}P_\theta(h)$, as claimed.

\bigskip

Finally, we show \eqref{env equ} -- the proof is the same as in the Monge-Amp\`ere case, but we reproduce the details for the readers convenience (see \cite{Tosatti18}). First, we show that $(\theta + i\p\pbar P(h))^m\wedge\omega^{n-m}$ vanishes outside the contact set $K := \{P_\theta(h) = h\}$ -- this is classical when $X$ is a K\"ahler manifold, and the same proof can be adapted to the Hermitian case \cite[Theorem 4.1]{GN18}. Here however, we give a different proof, utilizing the estimates we have already show for the $u_\e$. Note that, by equation \eqref{envelope CHE 2}, we have that $(\theta + i\p\pbar u_\e)^m\wedge\omega^{n-m} \rightarrow 0$ on the open set $X\setminus K$. Thus, we only need to show that the measures $(\theta+i\p\pbar u_\e)^m\wedge\omega^{n-m}$ converge to $(\theta+i\p\pbar P_\theta(h))^m\wedge\omega^{n-m}$.

By the uniform convergence of the $u_\e$ combined with the uniform lower bound for $i\p\pbar u_\e$, there exists a constant $C$ independent of $\e$ such that $u_\e, P_\theta(h)\in\PSH(X,C\omega)$. By \cite[Proposition 1.2]{KN15}, we have that
\[
(C\omega+i\p\pbar u_\e)^{i} \rightarrow (C\omega+i\p\pbar P_\theta(h))^{i},
\]
for each $1\leq i \leq n$. Writing then:
\[
\begin{split}
(\theta+i\p\pbar u_\e)^m\wedge\omega^{n-m} & = (\theta-C\omega+C\omega+i\p\pbar u_\e)^m\wedge\omega^{n-m} \\
& = \sum_{i=0}^{m}\left(
\begin{matrix}
m\\
i
\end{matrix}
\right)(C\omega+i\p\pbar u_\e)^{i}\wedge(\theta-C\omega)^{m-i}\wedge\omega^{n-m},
\end{split}
\]
we immediately see
\[
(\theta+i\p\pbar u_\e)^m\wedge\omega^{n-m} \rightarrow (\theta+i\p\pbar P_\theta(h))^m\wedge\omega^{n-m},
\]
as needed.

We now deal with the measure on the contact set. Since $P_\theta(h) - h \in C^1(X)$, we have that $\nabla(P_\theta(h) - h) = 0$ on the closed set $K$. By the first half of the proof, $P_\theta(h) \in C ^{1,1}(X)$, so we have that $\nabla_i (P_\theta(h) - h)$ is actually Lipschitz for any $i = 1,\ldots, 2n$ (working in a local coordinate chart, with the $i$ being real indices). It follows that:
\[
\nabla\nabla_i (P_\theta(h) - h) = 0
\]
almost everywhere on $\{\nabla_i(P_\theta(h) - h) = 0\} \supseteq K$ (see \cite[Theorem 3.2.6]{AmbTil}, for instance). Therefore, we have $\nabla^2 P_\theta(h) = \nabla^2 h$ a.e. on $K$, and so $\theta + i\p\pbar P_\theta(h) = \theta + i\p\pbar h$ on $K$, as the complex deriviatives are just complex linear combinations of the real ones. This establishes \eqref{env equ}.
\end{proof}



\begin{thebibliography}{99}

\bibitem{AmbTil} Ambrosio, L., Tilli, P. {\em Topics on analysis in metric spaces}, Oxford Lecture Series in Mathematics and its Applications, {\bf 25}, Oxford University Press, Oxford, 2004.

\bibitem{Andrews94} Andrews, B. {\em Contraction of convex hypersurfaces in Euclidean space}, Calc. Var. Partial Differential Equations {\bf 2} (1994), no. 2, 151--171.

\bibitem{Aubin76} Aubin, T. {\em \'Equations du type Monge-Amp\`ere sur les vari\'et\'es k\"ahleriennes compactes}, C. R. Acad. Sci. Paris S\'er. A-B 283 (1976), no. 3, Aiii, A119--A121.

\bibitem{Berman09} Berman, R.J. {\em Bergman kernels and equilibrium measures for line bundles over projective manifolds}, Amer. J. Math. {\bf 131} (2009), no. 5, 14851524.

\bibitem{Berman17} Berman, R.J. {\em On the optimal regularity of weak geodesics in the space of metrics on a polarized manifold}, Analysis meets geometry, 111--120,
Trends Math., Birkh\"auser/Springer, Cham, 2017.

\bibitem{Berman19} Berman, R.J. {\em From Monge-Amp\`ere equations to envelopes and geodesic rays in the zero temperature limit}, Math. Z. {\bf 291} (2019), no. 1--2, 365--394.

\bibitem{BD} Berman, R.J., Demailly, J.-P. {\em Regularity of plurisubharmonic upper envelopes in big cohomology classes}, In Perspectives in analysis, geometry, and topology, {\bf 296}, 39--66, Birkh\"auser/Springer, New York, 2012.

\bibitem{Blocki12} B\l ocki, Z. {\em On geodesics in the space of K\"ahler metrics}, in {\em Advances in geometric analysis}, 3--19, Adv. Lect. Math. (ALM), 21, Int. Press, Somerville, MA, 2012.

\bibitem{CNS85} Caffarelli, L., Nirenberg, L., Spruck, J. {\em The Dirichlet problem for nonlinear second-order elliptic equations. III. Functions of the eigenvalues of the Hessian}, Acta Math. {\bf 155} (1985), no. 3-4, 261--301.

\bibitem{Calabi57} Calabi, E. {\em On K\"ahler manifolds with vanishing canonical class}, Algebraic geometry and topology. A symposium in honor of S. Lefschetz, pp. 78--89. Princeton University Press, Princeton, N. J., 1957.

\bibitem{Chen00} Chen, X.X. {\em The space of K\"ahler metrics}, J. Differential Geom. {\bf 56} (2000), no. 2, 189--234.

\bibitem{Ch2} Chen, X.X. {\em On the lower bound of the Mabuchi energy and its application}, Int. Math. Res. Notices 12 (2000), 607--623

\bibitem{Cherrier87} Cherrier, P. {\em \'Equations de Monge-Amp\`ere sur les vari\'et\'es Hermitiennes compactes}, Bull. Sc. Math. (2) {\bf 111} (1987), 343--385.

\bibitem{CW01} Chou, K.-S., Wang, X.-J. {\em A variational theory of the Hessian equation}, Comm. Pure Appl. Math. {\bf 54} (2001), no. 9, 1029--1064.

\bibitem{Chu18} Chu, J. {$C^{1,1}$ regularity of degenerate complex Monge-Amp\`{e}re equations and some applications,} to appear in Anal. PDE.

\bibitem{CM19} Chu, J., McCleerey, N. {\em $C^{1,1}$ regularity of geodesics of singular K\"ahler metrics}, preprint, arXiv: 1901.02105.

\bibitem{CTW17} Chu, J., Tosatti, V., Weinkove, B. {\em On the $C^{1,1}$ regularity of geodesics in the space of K\"ahler metrics}, Ann. PDE {\bf 3} (2017), no. 2, Paper No. 15, 12 pp.

\bibitem{CTW18} Chu, J., Tosatti, V., Weinkove, B. {\em $C^{1,1}$ regularity for degenerate complex Monge-Amp\`ere equations and geodesic rays}, Comm. Partial Differential Equations {\bf 43} (2018), no. 2, 292--312.

\bibitem{CTW19} Chu, J., Tosatti, V., Weinkove, B. {\em The Monge-Amp\`ere equation for non-integrable almost complex structures}, J. Eur. Math. Soc. (JEMS) {\bf 21} (2019), no. 7, 1949--1984.

\bibitem{CZ19} Chu, J., Zhou, B. {\em Optimal regularity of plurisubharmonic envelopes on compact Hermitian manifolds}, Sci. China Math. {\bf 62} (2019), no. 2, 371--380.

\bibitem{Da14} Darvas, T. {\em Morse theory and geodesics in the space of K\"ahler metrics}, Proc. Amer. Math. Soc. {\bf 142} (2014), no. 8, 2775--2782.

\bibitem{DL12} Darvas, T., Lempert, L. {\em Weak geodesics in the space of K\"ahler metrics}, Math. Res. Lett. {\bf 19} (2012), no. 5, 1127-1135.

\bibitem{DR16} Darvas, T. Rubenstein, Y. {\em Kiselman’s principle, the Dirichlet problem for the Monge-Amp\`ere equation, and rooftop obstacle problems}, J. Math. Soc. Japan {\bf 68} (2016), no. 2, 773--796.

\bibitem{DNT19} Di Nezza, E., Trapani, S. {\em Monge-Amp\`ere measures on contact sets}, preprint, arXiv: 1912.12720.

\bibitem{DK17} Dinew, S., Ko\l odziej, S. {\em Liouville and Calabi-Yau type theorems for complex Hessian equations}, Amer. J. Math. {\bf 139} (2017), no. 2, 403--415.

\bibitem{DPZ19} Dinew, S., Pli\'s, S., Zhang, X. {\em Regularity of degenerate Hessian equations}, Calc. Var. Partial Differential Equations {\bf 58} (2019), no. 4, Paper No. 138, 21 pp.

\bibitem{Do1} Donaldson, S.K. {\em Moment maps and diffeomorphisms}, Asian J. Math. 3, no. 1 (1999), 1--16.

\bibitem{Donaldson99} Donaldson, S.K. {\em Symmetric spaces, K\"ahler geometry and Hamiltonian dynamics}, in {\em Northern California Symplectic Geometry Seminar}, 13--33, Amer. Math. Soc., Providence, RI, 1999.

\bibitem{EH89} Ecker, K., Huisken, G. {\em Immersed hypersurfaces with constant Weingarten curvature}, Math. Ann. {\bf 283} (1989), no. 2, 329--332.

 \bibitem{FLM11} Fang, H., Lai, M., Ma, X.-N. {\em On a class of fully nonlinear flows in K\"ahler geometry}, J. Reine Angew. Math. {\bf 653} (2011), 189--220.

\bibitem{FWW10} Fu, J., Wang, Z., Wu, D. {\em Form-type Calabi-Yau equations}, Math. Res. Lett. {\bf 17} (2010), no. 5, 887--903.

\bibitem{FWW15} Fu, J., Wang, Z., Wu, D. {\em Form-type Calabi-Yau equations on K\"ahler manifolds of nonnegative orthogonal bisectional curvature},  Calc. Var. Partial Differential Equations {\bf 52} (2015), no. 1-2, 327--344.

\bibitem{Gerhardt96} Gerhardt, C. {\em Closed Weingarten hypersurfaces in Riemannian manifolds}, J. Differential Geom. {\bf 43} (1996), no. 3, 612--641.

\bibitem{GN18} Gu, D., Nguyen, N.C. {\em The Dirichlet problem for a complex Hessian equation on compact Hermitian manifolds with boundary}, Ann. Sc. Norm. Super. Pisa Cl. Sci. (5) {\bf 18} (2018), no. 4, 1189--1248.

\bibitem{Guan14} Guan, B. {\em Second-order estimates and regularity for fully nonlinear elliptic equations on Riemannian manifolds}, Duke Math. J. {\bf 163} (2014), no. 8, 1491--1524.

\bibitem{GL10} Guan, B., Li, Q. {\em Complex Monge-Amp\`ere equations and totally real submanifolds}, Adv. Math. {\bf 225} (2010), no. 3, 1185--1223.

\bibitem{GS15} Guan, B., Sun, W. {\em On a class of fully nonlinear elliptic equations on Hermitian manifolds}, Calc. Var. Partial Differential Equations {\bf 54} (2015), no. 1, 901--916.

\bibitem{Hanani96} Hanani, A. {\em \'{E}quations du type de Monge-Amp\`{e}re sur les vari\'{e}t\'{e}s hermitiennes compactes}, J. Funct. Anal. {\bf 137} (1996), no.1, 49--75.

\bibitem{HL13} Harvey, F.R., Lawson, H.B. {\em The equivalence of viscosity and distributional subsolutions for convex subequations--a strong Bellman principle}, Bull. Braz. Math. Soc. (N.S.) {\bf 44} (2013), no. 4, 621--652.

\bibitem{HMW10} Hou, Z., Ma, X.-N., Wu, D. {\em A second order estimate for complex Hessian equations on a compact K\"ahler manifold}, Math. Res. Lett. {\bf 17} (2010), no. 3, 547--561.

\bibitem{KN15} Ko\l odziej, S., Nguyen, N.C. {\em Weak solutions to the complex Monge-Amp\`ere equation on Hermitian manifolds}, Analysis, complex geometry, and mathematical physics: in honor of Duong H. Phong, 141--158, Contemp. Math., 644, Amer. Math. Soc., Providence, RI, 2015.

\bibitem{KN16} Ko\l odziej, S., Nguyen, N.C. {\em Weak solutions of complex Hessian equations on compact Hermitian manifolds}, Compos. Math. {\bf 152} (2016), no. 11, 2221--2248.

\bibitem{LS15} Lejmi, M., Sz\'ekelyhidi, G. {\em The J-flow and stability}, Advances in Mathematics {\bf 274} (2015), 404--431

\bibitem{LV13} Lempert, L., Vivas, L. {\em Geodesics in the space of K\"ahler metrics}, Duke Math. J. {\bf 162}, (2013), no. 7, 1369--1381.

\bibitem{Lu13} Lu, C.H. {\em Solutions to degenerate complex Hessian equations}, J. Math. Pures Appl. (9) {\bf 100} (2013), no. 6, 785--805.

\bibitem{LN15} Lu, C.H., Nguyen, V.-D. {\em Degenerate complex Hessian equations on compact K\"ahler manifolds}, Indiana Univ. Math. J. {\bf 64} (2015), no. 6, 1721--1745.

\bibitem{Li14} Li, Y. {\em A priori estimates for Donaldson's equation over compact Hermitian manifolds}, Calc. Var. Partial Differential Equations {\bf 50} (2014), no. 3-4, 867--882.

\bibitem{Mabuchi87} Mabuchi, T. {\em Some symplectic geometry on compact K\"ahler manifolds. I}, Osaka J. Math. {\bf 24} (1987), no. 2, 227--252.

\bibitem{Michelsohn82} Michelsohn, M.L. {\em On the existence of special metrics in complex geometry}, Acta Math. {\bf 149} (1982), no. 3-4, 261--295.

\bibitem{Plis05} Pli\'s, S. {\em A counterexample to the regularity of the degenerate complex Monge-Amp\`ere equation}, Ann. Polon. Math. {\bf 86} (2005), no. 2, 171--175.

\bibitem{Plis13} Pli\'s, S. {\em The smoothing of $m$-subharmonic functions}, preprint, arXiv: 1312.1906.

\bibitem{Semmes92} Semmes, S. {\em Complex Monge-Amp\`ere and symplectic manifolds}, Amer. J. Math. {\bf 114} (1992), no. 3, 495--550.

\bibitem{SW08} Song, J., Weinkove, B. {\em On the convergence and singularities of the $J$-flow with applications to the Mabuchi energy}, Comm. Pure Appl. Math. {\bf 61} (2008), no. 2, 210--229.

\bibitem{Sun13} Sun, W. {\em On a class of fully nonlinear elliptic equations on closed Hermitian manifolds}, preprint, arXiv:1310.0362.

\bibitem{Sun17a} Sun, W. {\em On a class of fully nonlinear elliptic equations on closed Hermitian manifolds II: $L^{\infty}$ estimate}, Comm. Pure Appl. Math. {\bf 70} (2017), no. 1, 172--199.

\bibitem{Sun17b} Sun, W. {\em On uniform estimate of complex elliptic equations on closed Hermitian manifolds}, Commun. Pure Appl. Anal. {\bf 16} (2017), no. 5, 1553--1570.

\bibitem{Spruck05} Spruck, J. {\em Geometric aspects of the theory of fully nonlinear elliptic equations}, Global theory of minimal surfaces, Amer. Math. Soc., Providence, RI, 2005, 283--309.

\bibitem{Szekelyhidi18} Sz\'ekelyhidi, G. {\em Fully non-linear elliptic equations on compact Hermitian manifolds}, J. Differential Geom. {\bf 109} (2018), no. 2, 337--378.

\bibitem{Tosatti18} Tosatti, V. {\em Regularity of envelopes in K\"ahler classes}, Math. Res. Lett. {\bf 25} (2018), no. 1, 281--289.

\bibitem{TW10a} Tosatti, V., Weinkove, B. {\em Estimates for the complex Monge-Amp\`ere equation on Hermitian and balanced manifolds}, Asian J. Math. {\bf 14} (2010), no. 1, 19--40.

\bibitem{TW10b} Tosatti, V., Weinkove, B. {\em The complex Monge-Amp\`ere equation on compact Hermitian manifolds}, J. Amer. Math. Soc. {\bf 23} (2010), no. 4, 1187--1195.

\bibitem{TW17} Tosatti, V. and Weinkove, B., {\em The Monge-Amp\`ere equation for $(n-1)$-plurisubharmonic functions on a compact K\"ahler manifold}, J. Amer. Math. Soc. {\bf 30} (2017), no. 2, 311--346.

\bibitem{TW19} Tosatti, V. and Weinkove, B. {\em Hermitian metrics, $(n-1, n-1)$ forms and Monge-Amp\`ere equations}, J. Reine Angew. Math. {\bf 755} (2019), 67--101.

\bibitem{Trudinger95} Trudinger, N.S. {\em On the Dirichlet problem for Hessian equations}, Acta Math. {\bf 175} (1995), no. 2, 151--164.

\bibitem{Yau78} Yau, S.-T. {\em On the Ricci curvature of a compact K\"ahler manifold and the complex Monge-Amp\`ere equation, I}, Comm. Pure Appl. Math. {\bf 31} (1978), no. 3, 339--411.

\bibitem{Zhang17} Zhang, D. {\em Hessian equations on closed Hermitian manifolds}, Pacific J. Math. {\bf 291} (2017), no. 2, 485--510.

\bibitem{ZZ11} Zhang, X., Zhang, X. {\em  Regularity estimates of solutions to complex Monge-Amp\`ere equations on Hermitian manifolds}, J. Funct. Anal. {\bf 260} (2011), no. 7, 2004--2026.

\end{thebibliography}
\end{document}